\documentclass[12pt]{amsart}
\usepackage{amsmath,amssymb,amsthm,mathrsfs}

\usepackage[backref,pagebackref]{hyperref}
\usepackage[alphabetic,lite]{amsrefs}

\hypersetup{
  hidelinks,
  pdftitle={Local cohomology with Schubert support on compact Hermitian symmetric spaces},
  pdfauthor={Michael Perlman}
}

\usepackage{longtable}
\usepackage{capt-of}

\usepackage{tikz}

\newdimen\unitsize
\renewcommand{\a}{\underline{a}}
\renewcommand{\b}{\underline{b}}
\newcommand{\mc}[1]{\mathcal{#1}}
\newcommand{\DD}{\mathbb{D}}
\newcommand{\defi}[1]{{\upshape\sffamily #1}}

\newtheorem{thmx}{Theorem}

\newtheorem{theorem}{Theorem}[section]
\newtheorem{lemma}[theorem]{Lemma}
\newtheorem{prop}[theorem]{Proposition}
\newtheorem{cor}[theorem]{Corollary}

\theoremstyle{definition}
\newtheorem{example}[theorem]{Example}
 \newtheorem{remark}[theorem]{Remark}

\numberwithin{equation}{section}

\newcommand{\g}{\mathfrak g}
\newcommand{\m}{\mathfrak m}
\newcommand{\p}{\mathfrak p}
\newcommand{\mf}{\mathfrak}

\newcommand{\W}{\mathscr W}

\author{Michael Perlman}
\address{Department of Mathematics, The University of Alabama, Tuscaloosa, AL 35401}
\email{mperlman@ua.edu}

\title[Local cohomology on Hermitian symmetric spaces]{Local cohomology with Schubert support on compact Hermitian symmetric spaces}


\begin{document}

\begin{abstract}
We consider local cohomology sheaves on compact Hermitian symmetric spaces supported in Schubert varieties, which carry natural structures as mixed Hodge modules. Our main result gives an explicit root-theoretic construction, uniform across Lie types, for the $\mathscr{D}$-simple composition factors and weight filtration on these modules. For Lagrangian and orthogonal Grassmannians, we translate these formulas into combinatorial rules involving strict partitions and shifted Dyck patterns, analogous to existing formulas on ordinary Grassmannians. As an application, we expand upon work of Raicu--Weyman to describe the weight filtration on local cohomology supported in symmetric determinantal and Pfaffian varieties. We also give formulas for the Hodge rational homology level and local cohomological defect of every Schubert variety in a compact Hermitian symmetric space. Finally, we tabulate the local cohomology weight filtrations for all Schubert varieties in the two exceptional Hermitian pairs.
\end{abstract}

\maketitle

\section{Introduction}

\subsection{Overview} Given a closed subvariety $Z$ of a smooth complex variety $X$,
the local cohomology sheaves $\mathscr H^q_Z(\mathscr O_X)$ are fundamental objects of study in commutative algebra and algebraic geometry. These sheaves carry natural structures as holonomic $\mathscr D_X$-modules, where $\mathscr D_X$ is the
sheaf of algebraic differential operators on $X$. In particular, they have finite length over $\mathscr D_X$, and a basic problem
is to determine their simple composition factors. Moreover, Saito's theory \cite{saito90} functorially endows each $\mathscr H^q_Z(\mathscr O_X)$ with structures as mixed Hodge modules, implying that they have two increasing filtrations of geometric origin: (1) the Hodge filtration $F_{\bullet}$, an infinite filtration by coherent $\mathscr{O}_X$-modules, and (2) the weight filtration $W_{\bullet}$, a finite filtration by holonomic $\mathscr{D}_X$-modules with semi-simple quotients. 

The Hodge and weight filtrations on $\mathscr{H}^q_Z(\mathscr{O}_X)$ have been a subject of considerable interest in recent past, in relation to singularities of $Z$ (see, for instance, \cites{MP1,MP4,CDMO,ParkPopa,DORhrh}). For local complete intersections, comparison of the Hodge and
order filtrations characterizes higher Du Bois singularities
\cite{MP4}; see also \cites{MOPW,JKSY,FL,MPkr} for higher Du Bois and higher rational singularities. Despite these developments, much less is known for varieties that are not local complete intersections, and explicit calculations of weight filtrations are limited to a few families, including secant varieties of smooth projective varieties \cite{CDORsecant}, determinantal varieties \cites{raicu2016local,iterated,MHM}, Schubert varieties in Grassmannians \cite{schubertLC}, and certain affine toric varieties \cite{KVlocal}.

In this article, we determine the simple composition factors and
weight filtration on local cohomology supported in any Schubert
variety in a compact Hermitian symmetric space. Our formula is root-theoretic and uniform across Lie types, extending the
Grassmannian formulas of \cite{schubertLC} to Lagrangian and
orthogonal Grassmannians, smooth quadrics, and the two exceptional
Hermitian symmetric spaces.

Let $G$ be a simple complex algebraic group, and let $P\subseteq G$
be a maximal parabolic subgroup with abelian unipotent radical.
Then $X=G/P$ is a compact Hermitian symmetric space. Writing
$\g=\operatorname{Lie}(G)$ and $\m$ for the Lie algebra of a Levi
factor of $P$, we call $(\g,\m)$ a \defi{Hermitian pair}.

\noindent The Hermitian pairs $(\g,\m)$ are classified by their underlying root systems:
\[
(\mathsf{A}_{n-1},\mathsf{A}_{k-1}\times \mathsf{A}_{n-k-1}),\quad
(\mathsf{B}_n,\mathsf{B}_{n-1}),\quad (\mathsf{C}_n,\mathsf{A}_{n-1}),\quad (\mathsf{D}_n,\mathsf{A}_{n-1}),\quad (\mathsf{D}_n,\mathsf{D}_{n-1}),\quad (\mathsf{E}_6,\mathsf{D}_5),\quad (\mathsf{E}_7,\mathsf{E}_6).
\]
The pairs
$(\mathsf A_{n-1},\mathsf A_{k-1}\times\mathsf A_{n-k-1})$,
$(\mathsf C_n,\mathsf A_{n-1})$, and
$(\mathsf D_n,\mathsf A_{n-1})$ correspond, respectively, to
the Grassmannian $\operatorname{Gr}(k,n)$, the Lagrangian
Grassmannian $\operatorname{LGr}(n,2n)$, and the orthogonal Grassmannian
$\operatorname{OGr}(n,2n)$.
The pairs $(\mathsf B_n,\mathsf B_{n-1})$ and
$(\mathsf D_n,\mathsf D_{n-1})$ give the smooth quadrics
$Q^{2n-1}$ and $Q^{2n-2}$, where
$Q^r\subseteq\mathbb P^{r+1}$.
The pairs $(\mathsf E_6,\mathsf D_5)$ and
$(\mathsf E_7,\mathsf E_6)$ give the Cayley plane and
the Freudenthal variety.

In the following statement, we write $\mathscr O_X^H$ for the trivial Hodge module on $X$ and $\operatorname{IC}_Z^H(k)$ for the $k$-th Tate twist of the intersection cohomology Hodge module associated to a subvariety $Z\subseteq X$.

\begin{thmx}\label{thm:intro-local-cohomology}
Let $X=G/P$ be a compact Hermitian symmetric space.

There is an explicit root-theoretic construction, uniform across
Lie types, which assigns to each Schubert variety
$Z_x\subseteq X$ and integers $q\geq0$, $p\in\mathbb Z$,
a finite index set $\mathscr B_x(q,p)$ and an assignment
$\xi\mapsto Z_\xi$ of Schubert subvarieties of $Z_x$,
not necessarily distinct, such that
\[
\operatorname{Gr}^W_p
\mathscr H^q_{Z_x}(\mathscr O_X^H)
\cong
\bigoplus_{\xi\in\mathscr B_x(q,p)}
\operatorname{IC}_{Z_\xi}^H
\left(\frac{\dim Z_\xi-p}{2}\right).
\]
\end{thmx}

\noindent As an application, we give formulas for the Hodge rational homology level \cites{ParkPopa,DORhrh} and local cohomological defect of every Schubert variety in a compact Hermitian symmetric space. 

 Our construction, stated precisely in Theorem~\ref{thm:local-cohomology}, is described in terms of the admissible sets of pairwise orthogonal roots (see Section \ref{sec:Verma}), which were introduced by Enright--Shelton \cite{EnrightShelton} and used by Collingwood--Irving--Shelton
\cite{cisWeight} to describe Loewy filtrations on parabolic Verma modules. For now, we summarize as follows: for each Weyl group index $y\geq x$, we consider sets $\Omega$ of roots admissible for $y$, and write $\Omega_y^+$ for the
roots in $\Omega$ sent to positive roots by $y$.
Our formula retains precisely those pairs $(y,\Omega)$ for which
$\Omega_y^+$ contains no simple root and
$\Omega\cup\{\alpha\}$ is not admissible for $ys_\alpha$, for all simple roots $\alpha$ with $x\leq ys_\alpha<y$. Such a pair $(y,\Omega)$ belongs to $\mathscr B_x(q,p)$ if $q=\ell(y)$ and $p=\dim X+\ell(y)+|\Omega_y^+|$. The assignment $(y,\Omega)\mapsto Z_{(y,\Omega)}$ is determined by the minimal coset representative of
$y\prod_{\gamma\in\Omega_y^+}s_\gamma$. 

In Section \ref{sec:shifted-shapes}, Theorem \ref{thm:intro-local-cohomology} is translated into the language of shifted Young diagrams in the case of Lagrangian
and orthogonal Grassmannians. Upon restriction to the opposite big cells, our results give, for the first time, the weight filtration for rank varieties of symmetric and skew-symmetric matrices, in all sizes and ranks. We recover
the composition factor formulas of Raicu--Weyman
\cite{rw,raicu2016local}, and extend the symmetric determinant and
Pfaffian hypersurface calculations of \cite{LYsemi}. 

The proof of Theorem \ref{thm:intro-local-cohomology} uses the Grothendieck--Cousin complex \cite{kempf}, viewed as a
complex of mixed Hodge modules \cite{schubertLC}.
The new step is a uniform description of the images of adjacent
maps between parabolic Verma modules
(Theorem~\ref{thm:adjacentmaps}). This yields direct sum decompositions of weighted pieces of the Grothendieck--Cousin complexes
into Koszul complexes, from which we obtain Theorem \ref{thm:intro-local-cohomology}. In \cite{schubertLC}, the corresponding statement was proved using superduality and the known submodule structure of Kac modules for the general linear Lie superalgebra. Here we give a more elementary proof using only the known composition factors of parabolic Verma modules and translation functors. The argument is intrinsic to parabolic category $\mathcal O$ and applies uniformly to all Hermitian pairs.

\subsection*{Organization}  Section \ref{sec:shifted-shapes} is a statement of results in the cases $(\mathsf C_n,\mathsf A_{n-1})$ and
$(\mathsf D_n,\mathsf A_{n-1})$, which are proven and applied to symmetric determinantal and Pfaffian varieties in Section~\ref{sec:shifted-pairs}.  Section~\ref{sec:preliminaries} collects the necessary preliminaries on Lie algebras, mixed Hodge modules, and localization. We prove Theorem~\ref{thm:adjacentmaps} on maps of Verma modules in Section~\ref{sec:adjacent-proof} and deduce the weight filtration on local cohomology, Hodge rational homology level, and local cohomological defect in Section~\ref{sec:local-cohomology}. Sections~\ref{sec:quadrics} and~\ref{sec:exceptional-pairs} treat quadrics and the exceptional Hermitian symmetric spaces, respectively.

\subsection{The Lagrangian and orthogonal Grassmannians}\label{sec:shifted-shapes}

We give a combinatorial form of the local cohomology
formula for types $(\mathsf C_n,\mathsf A_{n-1})$ and
$(\mathsf D_n,\mathsf A_{n-1})$, corresponding to the Lagrangian Grassmannian
$X=\operatorname{LGr}(n,2n)$ and the orthogonal Grassmannian $X=\operatorname{OGr}(n,2n)$. Set $m=n$ in the first case and $m=n-1$ in the second.
In either case, $d_X=\dim X=m(m+1)/2$. 

Schubert varieties in $X$ are parametrized by \defi{strict partitions} (see, for instance \cite{EHP}), which are sequences
$\a=(a_1,\cdots,a_k)$ of integers satisfying
\begin{equation}\label{eq:strict}
m\geq a_1>a_2>\cdots>a_k>0.
\end{equation}
The corresponding Schubert variety $Z_{\a}$ has dimension
\[
\dim Z_{\a}=|\a|=a_1+\cdots +a_k,
\]
and given two partitions $\a$ and $\b$ satisfying (\ref{eq:strict}) we have
\[
Z_{\b}\subseteq Z_{\a}\quad \iff \quad \b\subseteq \a,
\]
where $\b\subseteq\a$ means that $b_i\leq a_i$ for all $i\geq1$,
after extending both partitions by zero parts.

We identify a strict partition with its \defi{shifted diagram}
\[
 \a\quad\longleftrightarrow\quad
 \{(i,j)\in\mathbb Z_{>0}^2\mid
       1\leq i\leq k,\ i\leq j\leq i+a_i-1\}.
\]
In words, the diagram has $a_i$ boxes in row $i$, and row $i$ begins
in column $i$. Given two strict partitions $\b\subseteq \a$ the corresponding \defi{skew shifted shape} is $\a/\b:=\a\setminus\b$. For example, the following are the shifted diagram $(6,3,1)$ and the
skew shifted shape $(6,3,1)/(3,1)$\\ 
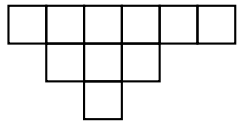
\begin{center}
\begin{minipage}{.35\textwidth}
\centering
\begin{tikzpicture}[x=\unitsize,y=\unitsize]
\foreach \shrow/\shlen in {1/6,2/3,3/1}{
  \pgfmathtruncatemacro{\shlast}{\shrow+\shlen-1}
  \foreach \shcol in {\shrow,...,\shlast}{
    \draw[thick] ({2*(\shcol-1)},{-2*(\shrow-1)}) rectangle ++(2,-2);
  }
}
\end{tikzpicture}
\par\smallskip $\a=(6,3,1)$
\end{minipage}
\qquad
\begin{minipage}{.35\textwidth}
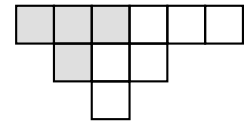
\centering
\begin{tikzpicture}[x=\unitsize,y=\unitsize]
\foreach \shrow/\shlen in {1/3,2/1}{
  \pgfmathtruncatemacro{\shlast}{\shrow+\shlen-1}
  \foreach \shcol in {\shrow,...,\shlast}{
    \fill[gray!25] ({2*(\shcol-1)},{-2*(\shrow-1)}) rectangle ++(2,-2);
  }
}
\foreach \shrow/\shlen in {1/6,2/3,3/1}{
  \pgfmathtruncatemacro{\shlast}{\shrow+\shlen-1}
  \foreach \shcol in {\shrow,...,\shlast}{
    \draw[thick] ({2*(\shcol-1)},{-2*(\shrow-1)}) rectangle ++(2,-2);
  }
}
\end{tikzpicture}
\par\smallskip $\a/\b=(6,3,1)/(3,1)$
\end{minipage}
\end{center}

\medskip 

\noindent In the second diagram, the gray
boxes belong to $(3,1)$ and the unshaded boxes form the skew shape. We call a shape
\defi{connected} if any two of its boxes can be joined by a sequence
of boxes sharing edges. 

A \defi{path} $P$ in $\mathbb Z_{>0}^2$ is a collection of boxes
\[
 P=\{x_1,\ldots,x_t\},\qquad x_u=(i_u,j_u),\qquad t\geq1,
\]
satisfying, for $1\leq u<t$,
\[
 x_{u+1}=(i_u+1,j_u)\quad\textnormal{or}\quad
 x_{u+1}=(i_u,j_u-1).
\]
In words, a path is obtained by starting at $x_1$ and
walking in any combination of South and West. 

We define the \defi{level} of a box by
$\operatorname{lv}(i,j)=i+j$.
The path $P$ is called \defi{almost Dyck} if
\begin{equation}\label{eq:sh-almost-dyck}
 \operatorname{lv}(x_u)\leq\operatorname{lv}(x_1)
 \qquad\textnormal{for all }1\leq u\leq t.
\end{equation}
If, in addition, $\operatorname{lv}(x_t)=\operatorname{lv}(x_1)$,
then $P$ is called a \defi{Dyck path}. Thus an almost Dyck path
has no box below the antidiagonal through its first box; a Dyck
path also ends on this antidiagonal.
The following are a Dyck path, an almost Dyck path that is not
Dyck, and a path that is not almost Dyck:\\

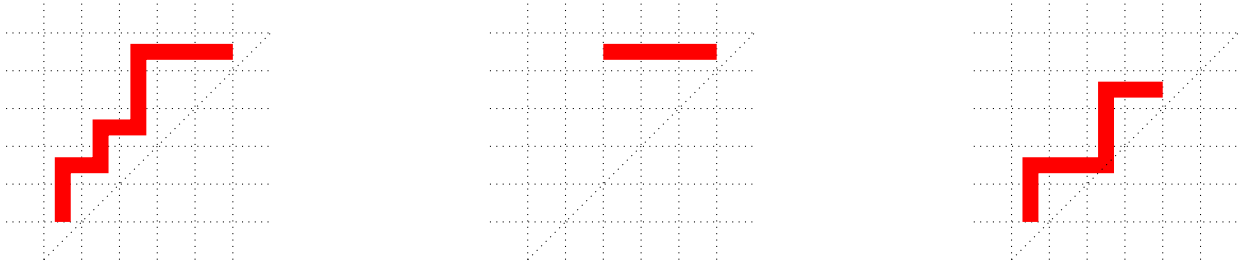
\begin{center}
\begin{minipage}{.30\textwidth}
\centering
\begin{tikzpicture}[x=\unitsize,y=\unitsize,baseline=0]
\tikzset{vertex/.style={}}%
\tikzset{edge/.style={  thick}}%
\draw[dotted] (0,0) -- (14,0);
\draw[dotted] (0,2) -- (14,2);
\draw[dotted] (0,4) -- (14,4);
\draw[dotted] (0,6) -- (14,6);
\draw[dotted] (0,8) -- (14,8);
\draw[dotted] (0,10) -- (14,10);
\draw[dotted] (2,-2) -- (2,12);
\draw[dotted] (4,-2) -- (4,12);
\draw[dotted] (6,-2) -- (6,12);
\draw[dotted] (8,-2) -- (8,12);
\draw[dotted] (10,-2) -- (10,12);
\draw[dotted] (12,-2) -- (12,12);
\draw[dotted] (2,-2) -- (14,10);
\draw[red, line width=6pt] (3,0) -- (3,3) -- (5,3) -- (5,5) -- (7,5) -- (7,9) -- (12,9);
%\draw[dotted] (0,0) -- (8,8);
\end{tikzpicture}
\end{minipage}
\hfill
\begin{minipage}{.30\textwidth}
\centering
\begin{tikzpicture}[x=\unitsize,y=\unitsize,baseline=0]
\tikzset{vertex/.style={}}%
\tikzset{edge/.style={  thick}}%
\draw[dotted] (0,0) -- (14,0);
\draw[dotted] (0,2) -- (14,2);
\draw[dotted] (0,4) -- (14,4);
\draw[dotted] (0,6) -- (14,6);
\draw[dotted] (0,8) -- (14,8);
\draw[dotted] (0,10) -- (14,10);
\draw[dotted] (2,-2) -- (2,12);
\draw[dotted] (4,-2) -- (4,12);
\draw[dotted] (6,-2) -- (6,12);
\draw[dotted] (8,-2) -- (8,12);
\draw[dotted] (10,-2) -- (10,12);
\draw[dotted] (12,-2) -- (12,12);
\draw[dotted] (2,-2) -- (14,10);
\draw[red, line width=6pt] (6,9) -- (12,9);
%\draw[dotted] (0,0) -- (8,8);
\end{tikzpicture}
\end{minipage}
\hfill
\begin{minipage}{.30\textwidth}
\centering
\begin{tikzpicture}[x=\unitsize,y=\unitsize,baseline=0]
\tikzset{vertex/.style={}}%
\tikzset{edge/.style={  thick}}%
\draw[dotted] (0,0) -- (14,0);
\draw[dotted] (0,2) -- (14,2);
\draw[dotted] (0,4) -- (14,4);
\draw[dotted] (0,6) -- (14,6);
\draw[dotted] (0,8) -- (14,8);
\draw[dotted] (0,10) -- (14,10);
\draw[dotted] (2,-2) -- (2,12);
\draw[dotted] (4,-2) -- (4,12);
\draw[dotted] (6,-2) -- (6,12);
\draw[dotted] (8,-2) -- (8,12);
\draw[dotted] (10,-2) -- (10,12);
\draw[dotted] (12,-2) -- (12,12);
\draw[red, line width=6pt] (3,0) -- (3,3) -- (7,3) -- (7,7) -- (10,7) ;
\draw[dotted] (2,-2) -- (14,10);
%\draw[dotted] (0,0) -- (8,8);
\end{tikzpicture}
\end{minipage}
\end{center}

\medskip

For a nonempty connected skew shifted shape $\eta$, its
\defi{border strip} is the path
\begin{equation}\label{eq:sh-inner-boundary}
 \theta(\eta)=\{(i,j)\in\eta\mid(i-1,j-1)\notin\eta\}.
\end{equation}
Thus $\theta(\eta)$ consists of the boxes on the Northwest border
of $\eta$. Let $\theta=\{x_1,\ldots,x_t\}$ be an almost Dyck path,
with $x_u=(i_u,j_u)$ listed from Northeast to Southwest, and set
\[
 r=\max\{u\mid\operatorname{lv}(x_u)=\operatorname{lv}(x_1)\}.
\]
We define its \defi{shifted completion} (see \cite{brenti}) by
\begin{equation}\label{eq:sh-completion}
 \theta_s=\theta\sqcup E(\theta),\qquad
 E(\theta)=\{(i_u+1,j_u+1)\mid r<u\leq t\}.
\end{equation}
In words, for each box of $\theta$ after its last visit to the
initial antidiagonal, add one box to its southeast. For $\eta=(7,6,4,2)/(4,3)$, the border strip
$\theta=\theta(\eta)$ and the added boxes $E(\theta)$ are shown below.\\

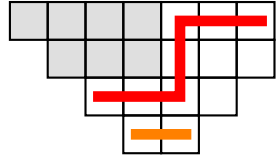
\begin{center}
\begin{tikzpicture}[x=\unitsize,y=\unitsize]
\foreach \shrow/\shlen in {1/4,2/3}{
  \pgfmathtruncatemacro{\shlast}{\shrow+\shlen-1}
  \foreach \shcol in {\shrow,...,\shlast}{
    \fill[gray!25] ({2*(\shcol-1)},{-2*(\shrow-1)})
      rectangle ++(2,-2);
  }
}
\foreach \shrow/\shlen in {1/7,2/6,3/4,4/2}{
  \pgfmathtruncatemacro{\shlast}{\shrow+\shlen-1}
  \foreach \shcol in {\shrow,...,\shlast}{
    \draw[thick] ({2*(\shcol-1)},{-2*(\shrow-1)})
      rectangle ++(2,-2);
  }
}
\draw[red,line width=4pt]
  (13.6,-1)--(9,-1)--(9,-5)--(4.4,-5);
\draw[orange,line width=4pt] (6.4,-7)--(9.6,-7);
\end{tikzpicture}
\par\smallskip
$r=5,\qquad E(\theta)=\{(4,4),(4,5)\},\qquad |\theta_s|=9.$
\end{center}

\medskip

Following \cite[Section 3]{brenti},
we define the \defi{shifted Dyck shapes} recursively as follows:
\begin{enumerate}
\item The empty shape is shifted Dyck.
\item A disconnected skew shifted shape is shifted Dyck if and only if
each of its connected components is shifted Dyck.
\item A nonempty connected skew shifted shape $\eta$ is shifted Dyck
if and only if its border strip $\theta=\theta(\eta)$ satisfies
\begin{equation}\label{eq:sh-dyck-recursion}
 \theta\textnormal{ is almost Dyck},\qquad
 \theta_s\subseteq\eta,\qquad
 |\theta_s\setminus\theta|\equiv0\pmod2,
\end{equation}
and the remaining skew shape $\eta\setminus\theta_s$ is shifted Dyck.
\end{enumerate}
Thus, one successively removes completed border strips,
applying the rule separately to the connected components that remain. The \defi{depth} $\operatorname{dp}(\eta)$ is the total number
of completed boundary strips removed by the recursion, counting
removals in different connected components separately.
We have $\operatorname{dp}(\varnothing)=0$.

Given a strict partition $\a$, an \defi{$\a$-admissible shifted Dyck
pattern} is a skew shape $\DD=\a/\b$ that is shifted Dyck.
We write
\begin{equation}\label{eq:sh-patterns}
 \operatorname{sDyck}(\a)=
 \{\a/\b\mid\b\subseteq\a,\ \a/\b\textnormal{ is shifted Dyck}\},
 \qquad \a^{\DD}=\b.
\end{equation}
The boundary removals are determined by the skew shape,
so the pattern is completely specified by its support $\a/\b$.
We say that the pattern has \defi{no singleton strips} if every
completed boundary strip removed in the recursion has at least
three boxes. The following shapes are shifted Dyck. \\
\begingroup
\newcommand{\shExampleGrid}[1]{%
  \foreach \shrow/\shlen in {#1}{%
    \pgfmathtruncatemacro{\shlast}{\shrow+\shlen-1}%
    \foreach \shcol in {\shrow,...,\shlast}{%
      \draw[thick] ({2*(\shcol-1)},{-2*(\shrow-1)})
        rectangle ++(2,-2);
    }%
  }%
}
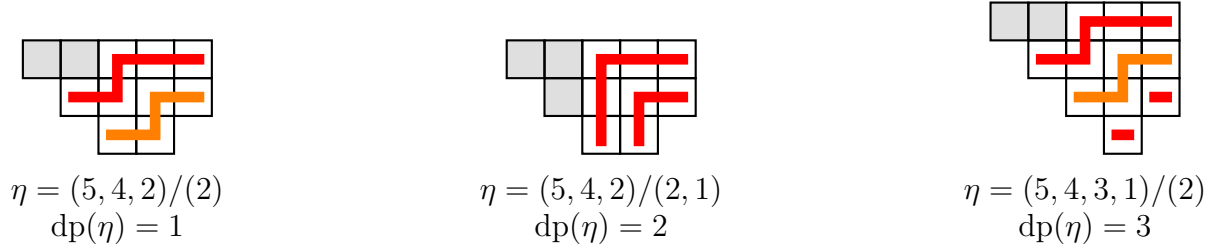
\begin{center}
\begin{minipage}[b]{.30\textwidth}
\centering
\begin{tikzpicture}[x=\unitsize,y=\unitsize]
\fill[gray!25] (0,0) rectangle (4,-2);
\shExampleGrid{1/5,2/4,3/2}
\draw[red,line width=4pt]
  (9.6,-1)--(5,-1)--(5,-3)--(2.4,-3);
\draw[orange,line width=4pt]
  (9.6,-3)--(7,-3)--(7,-5)--(4.4,-5);
\end{tikzpicture}
\par\smallskip
$\eta=(5,4,2)/(2)$\\
$\operatorname{dp}(\eta)=1$
\end{minipage}
\hfill
\begin{minipage}[b]{.30\textwidth}
\centering
\begin{tikzpicture}[x=\unitsize,y=\unitsize]
\fill[gray!25] (0,0) rectangle (4,-2);
\fill[gray!25] (2,-2) rectangle (4,-4);
\shExampleGrid{1/5,2/4,3/2}
\draw[red,line width=4pt] (9.6,-1)--(5,-1)--(5,-5.6);
\draw[red,line width=4pt] (9.6,-3)--(7,-3)--(7,-5.6);
\end{tikzpicture}
\par\smallskip
$\eta=(5,4,2)/(2,1)$\\
$\operatorname{dp}(\eta)=2$
\end{minipage}
\hfill
\begin{minipage}[b]{.30\textwidth}
\centering
\begin{tikzpicture}[x=\unitsize,y=\unitsize]
\fill[gray!25] (0,0) rectangle (4,-2);
\shExampleGrid{1/5,2/4,3/3,4/1}
\draw[red,line width=4pt]
  (9.6,-1)--(5,-1)--(5,-3)--(2.4,-3);
\draw[orange,line width=4pt]
  (9.6,-3)--(7,-3)--(7,-5)--(4.4,-5);
\draw[red,line width=4pt] (8.4,-5)--(9.6,-5);
\draw[red,line width=4pt] (6.4,-7)--(7.6,-7);
\end{tikzpicture}
\par\smallskip
$\eta=(5,4,3,1)/(2)$\\
$\operatorname{dp}(\eta)=3$
\end{minipage}
\end{center}

\medskip

\noindent The following shapes are not shifted Dyck:\\

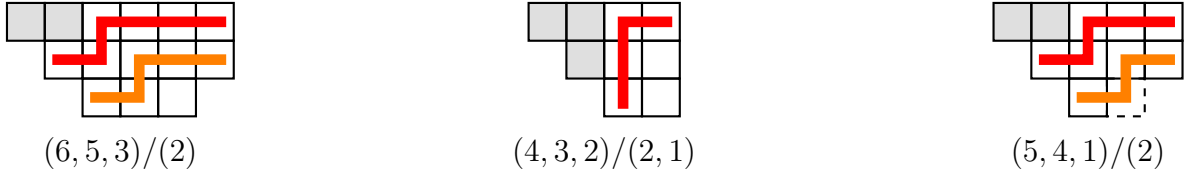
\begin{center}
\begin{minipage}[b]{.30\textwidth}
\centering
\begin{tikzpicture}[x=\unitsize,y=\unitsize]
\fill[gray!25] (0,0) rectangle (4,-2);
\shExampleGrid{1/6,2/5,3/3}
\draw[red,line width=4pt]
  (11.6,-1)--(5,-1)--(5,-3)--(2.4,-3);
\draw[orange,line width=4pt]
  (11.6,-3)--(7,-3)--(7,-5)--(4.4,-5);
\end{tikzpicture}
\par\smallskip $(6,5,3)/(2)$
\end{minipage}
\hfill
\begin{minipage}[b]{.30\textwidth}
\centering
\begin{tikzpicture}[x=\unitsize,y=\unitsize]
\fill[gray!25] (0,0) rectangle (4,-2);
\fill[gray!25] (2,-2) rectangle (4,-4);
\shExampleGrid{1/4,2/3,3/2}
\draw[red,line width=4pt] (7.6,-1)--(5,-1)--(5,-5.6);
\end{tikzpicture}
\par\smallskip $(4,3,2)/(2,1)$
\end{minipage}
\hfill
\begin{minipage}[b]{.30\textwidth}
\centering
\begin{tikzpicture}[x=\unitsize,y=\unitsize]
\fill[gray!25] (0,0) rectangle (4,-2);
\shExampleGrid{1/5,2/4,3/1}
\draw[thick,dashed] (6,-4) rectangle (8,-6);
\draw[red,line width=4pt]
  (9.6,-1)--(5,-1)--(5,-3)--(2.4,-3);
\draw[orange,line width=4pt]
  (9.6,-3)--(7,-3)--(7,-5)--(4.4,-5);
\end{tikzpicture}
\par\smallskip $(5,4,1)/(2)$
\end{minipage}
\end{center}

\medskip

\noindent The first completion adds five boxes, which is odd. The second border strip
is not almost Dyck. In the third, the completion leaves the skew shape.
\endgroup

Fix a strict partition $\a$. Given strict partitions
$\underline c\subseteq\b\subseteq\a$, set
\begin{equation}
 \eta=\b/\underline c,\qquad \mathbb B=\a/\b.
\end{equation}
We call the boxes of $\mathbb B$ \defi{bullets}, and call
$\DD=(\eta;\mathbb B)$ an \defi{$\a$-admissible augmented shifted
Dyck pattern} if the following conditions hold:
\begin{enumerate}
\item The skew shape $\eta$ is shifted Dyck.
\item For every $x\in\mathbb B$ such that $\b\cup\{x\}$ is a
shifted diagram, the skew shape
\begin{equation}\label{eq:sh-bullet-condition}
 (\b\cup\{x\})/\underline c=\eta\cup\{x\}
\end{equation}
is not shifted Dyck.
\end{enumerate}

The \defi{support}, \defi{remaining diagram}, \defi{depth}, and
\defi{number of bullets} of $\DD$ are, respectively,
\begin{align}
 \operatorname{supp}(\DD)&=\eta\sqcup\mathbb B=\a/\underline c,
 &\qquad \a^{\DD}&=\a\setminus\operatorname{supp}(\DD)=\underline c,
 \label{eq:sh-augmented-support}\\
 \operatorname{dp}(\DD)&=\operatorname{dp}(\eta),
 & b(\DD)&=|\mathbb B|=|\a|-|\b|.
 \label{eq:sh-augmented-statistics}
\end{align}
We write $\operatorname{sDyck}^{\bullet}(\a)$ for the set of
$\a$-admissible augmented shifted Dyck patterns. Patterns without
bullets are identified with $\operatorname{sDyck}(\a)$ via
$\eta\leftrightarrow(\eta;\varnothing)$. The empty pattern has
$\eta=\mathbb B=\varnothing$ and $\a^{\DD}=\a$.

For $q\in\mathbb Z_{\geq0}$ and $p\in\mathbb Z$, define
\begin{align}
 \mathscr A^{\mathrm{sh}}(\a;q)
 &=\left\{\DD=(\eta;\mathbb B)\in
       \operatorname{sDyck}^{\bullet}(\a)\ \middle|\
       \begin{gathered}
        \eta\textnormal{ has no singleton strips},\\
        b(\DD)=q+|\a|-d_X
       \end{gathered}\right\},\label{eq:sh-lc-patterns}\\
 \mathscr A^{\mathrm{sh}}_p(\a;q)
 &=\{\DD\in\mathscr A^{\mathrm{sh}}(\a;q)\mid
                  \operatorname{dp}(\DD)=p-q-d_X\}.
 \label{eq:sh-weight-patterns}
\end{align}

\noindent With this notation, we are ready to state the main result of this subsection.

\begin{thmx}\label{thm:sh-local-cohomology}
Let $X$ be either space above, and let $\a$ be a strict partition. For $q\geq0$
and $p\in\mathbb Z$, there is an isomorphism of pure
Hodge modules
\begin{equation}\label{eq:sh-weight-formula}
 \operatorname{Gr}^W_p\mathscr H^q_{Z_{\a}}(\mathscr O_X^H)
 \cong
 \bigoplus_{\DD\in\mathscr A^{\mathrm{sh}}_p(\a;q)}
 \operatorname{IC}_{Z_{\a^{\DD}}}^H
       \left(\frac{|\a^{\DD}|-p}{2}\right).
\end{equation}
\end{thmx}

 The proof is given in Section~\ref{sec:shifted-proof}, and in Section \ref{sec:matrix-applications} this formula is used to give the weight filtration for rank varieties of (skew-)symmetric matrices. In the next example, we write $\operatorname{IC}_{\b}$ for the intersection cohomology $\mathscr D_X$-module of $Z_{\b}$.

\begin{example}\label{ex:sh-lgr-seven}
Let $X=\operatorname{LGr}(7,14)$ and $\a=(7,6,4,3)$.
Then $d_X=28$, $|\a|=20$, and $Z_{\a}$ has codimension eight.
There are seven admissible augmented shifted Dyck patterns with
no singleton strips. \\
\begingroup
\newcommand{\shlcgrid}{%
\foreach \shrow/\shlen in {1/7,2/6,3/4,4/3}{%
  \pgfmathtruncatemacro{\shlast}{\shrow+\shlen-1}%
  \foreach \shcol in {\shrow,...,\shlast}{%
    \draw[thick] ({2*(\shcol-1)},{-2*(\shrow-1)}) rectangle ++(2,-2);
  }%
}%
}
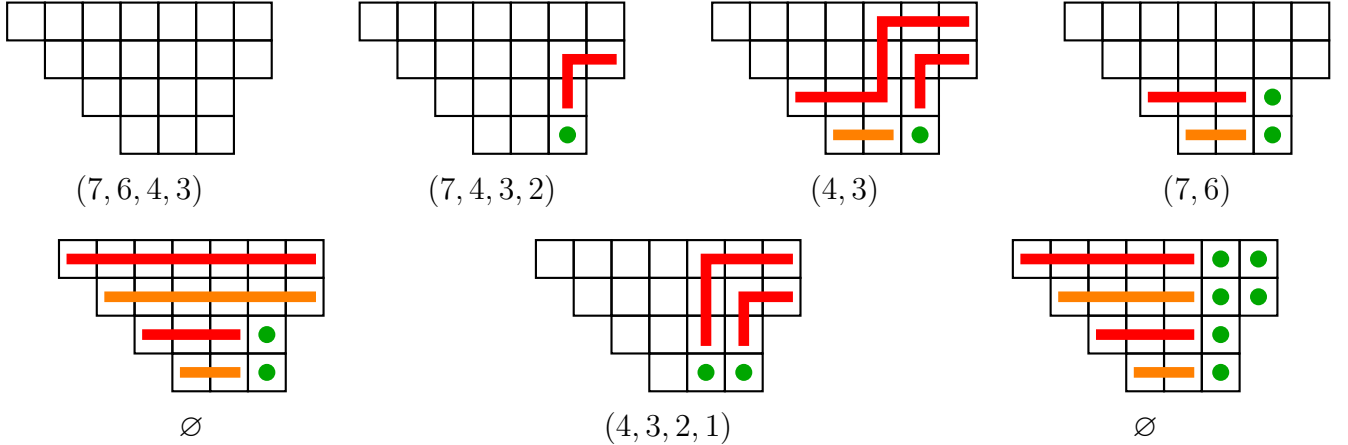
\begin{center}
\begin{minipage}[t]{.235\textwidth}
\centering
\begin{tikzpicture}[x=\unitsize,y=\unitsize]
\shlcgrid
\end{tikzpicture}
\par\smallskip $ (7,6,4,3)$
\end{minipage}%
\hfill
\begin{minipage}[t]{.235\textwidth}
\centering
\begin{tikzpicture}[x=\unitsize,y=\unitsize]
\shlcgrid
\draw[red,line width=4pt] (13.6,-3)--(11,-3)--(11,-5.6);
\fill[green!65!black] (11,-7) circle[radius=.45];
\end{tikzpicture}
\par\smallskip $(7,4,3,2)$
\end{minipage}%
\hfill
\begin{minipage}[t]{.235\textwidth}
\centering
\begin{tikzpicture}[x=\unitsize,y=\unitsize]
\shlcgrid
\draw[red,line width=4pt] (13.6,-1)--(9,-1)--(9,-5)--(4.4,-5);
\draw[orange,line width=4pt] (9.6,-7)--(6.4,-7);
\draw[red,line width=4pt] (13.6,-3)--(11,-3)--(11,-5.6);
\fill[green!65!black] (11,-7) circle[radius=.45];
\end{tikzpicture}
\par\smallskip $ (4,3)$
\end{minipage}%
\hfill
\begin{minipage}[t]{.235\textwidth}
\centering
\begin{tikzpicture}[x=\unitsize,y=\unitsize]
\shlcgrid
\draw[red,line width=4pt] (9.6,-5)--(4.4,-5);
\draw[orange,line width=4pt] (9.6,-7)--(6.4,-7);
\fill[green!65!black] (11,-5) circle[radius=.45];
\fill[green!65!black] (11,-7) circle[radius=.45];
\end{tikzpicture}
\par\smallskip $ (7,6)$
\end{minipage}%
\par\medskip
\medskip
\begin{minipage}[t]{.31\textwidth}
\centering
\begin{tikzpicture}[x=\unitsize,y=\unitsize]
\shlcgrid
\draw[red,line width=4pt] (13.6,-1)--(0.4,-1);
\draw[orange,line width=4pt] (13.6,-3)--(2.4,-3);
\draw[red,line width=4pt] (9.6,-5)--(4.4,-5);
\draw[orange,line width=4pt] (9.6,-7)--(6.4,-7);
\fill[green!65!black] (11,-5) circle[radius=.45];
\fill[green!65!black] (11,-7) circle[radius=.45];
\end{tikzpicture}
\par\smallskip $ \varnothing$
\end{minipage}%
\hfill
\begin{minipage}[t]{.31\textwidth}
\centering
\begin{tikzpicture}[x=\unitsize,y=\unitsize]
\shlcgrid
\draw[red,line width=4pt] (13.6,-1)--(9,-1)--(9,-5.6);
\draw[red,line width=4pt] (13.6,-3)--(11,-3)--(11,-5.6);
\fill[green!65!black] (9,-7) circle[radius=.45];
\fill[green!65!black] (11,-7) circle[radius=.45];
\end{tikzpicture}
\par\smallskip $ (4,3,2,1)$
\end{minipage}%
\hfill
\begin{minipage}[t]{.31\textwidth}
\centering
\begin{tikzpicture}[x=\unitsize,y=\unitsize]
\shlcgrid
\draw[red,line width=4pt] (9.6,-1)--(0.4,-1);
\draw[orange,line width=4pt] (9.6,-3)--(2.4,-3);
\draw[red,line width=4pt] (9.6,-5)--(4.4,-5);
\draw[orange,line width=4pt] (9.6,-7)--(6.4,-7);
\fill[green!65!black] (11,-1) circle[radius=.45];
\fill[green!65!black] (13,-1) circle[radius=.45];
\fill[green!65!black] (11,-3) circle[radius=.45];
\fill[green!65!black] (13,-3) circle[radius=.45];
\fill[green!65!black] (11,-5) circle[radius=.45];
\fill[green!65!black] (11,-7) circle[radius=.45];
\end{tikzpicture}
\par\smallskip $ \varnothing$
\end{minipage}%
\end{center}
\endgroup

\medskip 

\noindent The number of bullets determines the cohomological degree, after an initial shift by codimension, which is $c=8$ in this example. The first pattern belongs to $\mathscr{A}^{\mathrm{sh}}(\a;8)$, the second and third belong to $\mathscr{A}^{\mathrm{sh}}(\a;9)$, the next three belong to $\mathscr{A}^{\mathrm{sh}}(\a;10)$, and the seventh belongs to $\mathscr{A}^{\mathrm{sh}}(\a;14)$. Writing $[M]$ for the class of a holonomic $\mathscr D_X$-module in the Grothendieck group, Theorem~\ref{thm:sh-local-cohomology} gives
\[
 \big[\mathscr H^8_{Z_{\a}}(\mathscr O_X)\big]
   =\big[\operatorname{IC}_{(7,6,4,3)}\big],\quad
 \big[\mathscr H^9_{Z_{\a}}(\mathscr O_X)\big]
   =\big[\operatorname{IC}_{(7,4,3,2)}\big]
     +\big[\operatorname{IC}_{(4,3)}\big],
     \]
     \[
 \big[\mathscr H^{10}_{Z_{\a}}(\mathscr O_X)\big]
   =\big[\operatorname{IC}_{(7,6)}\big]
     +\big[\operatorname{IC}_{\varnothing}\big]
     +\big[\operatorname{IC}_{(4,3,2,1)}\big],\quad
 \big[\mathscr H^{14}_{Z_{\a}}(\mathscr O_X)\big]
   =\big[\operatorname{IC}_{\varnothing}\big],
\]
and all other local cohomology modules vanish. The graded pieces of the weight filtration are given by the number of completed paths $\theta_s$ plus the number of bullets, after an initial shift of $c+d_X$, which is $36$ in this example. Thus, the weights of the simple modules above are (from left to right):
\[
36,\, 38,\, 39,\, 39,\, 40,\, 40,\, 44.
\]
The same calculation applies to the partition
$(7,6,4,3)$ in a fixed component of $\operatorname{OGr}(8,16)$.\hfill\mbox{$\diamond$}
\end{example}

\section{Preliminaries}\label{sec:preliminaries}

\subsection{Mixed Hodge modules}

Let $X$ be a smooth complex variety of dimension $d_X$, with sheaf $\mathscr{D}_X$ of algebraic differential operators. We consider the category $\operatorname{MHM}(X)$ of algebraic mixed Hodge modules on $X$ (see \cite{saito90} or \cite[Section 8.3.3]{htt}). Part of the data of an object $M$ in $\operatorname{MHM}(X)$ is a pair $(\mathscr{M}, W_{\bullet})$, where $\mathscr{M}$ is a holonomic $\mathscr{D}_X$-module, known as the $\mathscr{D}_X$-module underlying $M$, and $W_{\bullet}$ is an increasing $\mathbb{Z}$-indexed filtration on $\mathscr{M}$ by $\mathscr{D}_X$-modules, known as the weight filtration on $M$. Morphisms in $\operatorname{MHM}(X)$ are strict with respect to the weight filtration in the sense that, for a morphism $\phi:M\to N$ of mixed Hodge modules, we have
\begin{equation}
\phi(W_k(\mathscr{M}))=\phi(\mathscr{M})\cap W_k(\mathscr{N}),    
\end{equation}
for all $k\in \mathbb{Z}$, where $(\mathscr{M},W_{\bullet})$ and $(\mathscr{N},W_{\bullet})$ are the filtered modules underlying $M$ and $N$.

The mixed Hodge module $M$ is pure of weight $w\in \mathbb{Z}$ if $\operatorname{Gr}_w^W(\mathscr{M})=\mathscr{M}$. The trivial Hodge module $\mathscr{O}_X^H$ has underlying $\mathscr{D}_X$-module $\mathscr{O}_X$ and weight filtration satisfying $\operatorname{Gr}_{d_X}^W(\mathscr{O}_X)=\mathscr{O}_X$, so is pure of weight $d_X$. Given a subvariety $Z\subseteq X$ of pure dimension $d_Z$, we consider the intersection cohomology Hodge module $\operatorname{IC}_Z^H$ \cite[Section 8.3.3(m13)]{htt}, which is pure of weight $d_Z$, and has underlying $\mathscr{D}_X$-module $\operatorname{IC}_Z$, the intersection cohomology module of Brylinski--Kashiwara \cite[Definition 3.4.1]{htt}. For $k\in \mathbb{Z}$, the Tate twist $\operatorname{IC}_Z^H(k)$ is pure of weight $d_Z-2k$ and has the same underlying $\mathscr{D}_X$-module.

Let $i:Y\hookrightarrow X$ be a locally closed subvariety of $X$. Given $M\in \operatorname{MHM}(X)$, we define the local cohomology modules of $M$ with support in $Y$ via
\begin{equation}
\mathscr{H}^q_Y(M)=H^q(i_{\ast}i^{!}M),\quad q\in \mathbb{Z},    
\end{equation}
where $i_{\ast}$ and $i^{!}$ are the direct image and pullback of mixed Hodge modules \cite[Section 4]{saito90}, lifting the $\mathscr{D}$-module functors $\int_i$ and $i^{\dagger}$ from \cite[Chapter 1.5]{htt}.

\subsection{Root systems and Weyl groups}

Let $\g$ be a simple complex Lie algebra with root system $\Phi$ and choice of positive system $\Phi^+$ with simple roots $\Pi \subseteq \Phi^+$ and Cartan and Borel subalgebras $\mathfrak{h}\subseteq \mathfrak{b}\subseteq  \g$. For roots $\gamma,\xi \in \Phi$ we use the partial order $\xi\leq\gamma$ for $\gamma-\xi\in\mathbb{Z}_{\geq 0}\Pi$. For a root $\beta
\in \Phi$, we write $\beta>0$ or $\beta<0$ to mean
$\beta\in\Phi^+$ or $\beta\in-\Phi^+$, respectively.

Let $\W$ denote the Weyl group of $\g$. Fix a $\W$-invariant positive-definite inner product
$\langle-,-\rangle$ on $\mathbb{R}\Phi$. For $\gamma\in\Phi$, let
$s_\gamma$ denote the reflection
\begin{equation}\label{eq:root-reflection}
s_\gamma(\lambda)
=
\lambda-\frac{2\langle\lambda,\gamma\rangle}
                 {\langle\gamma,\gamma\rangle}\gamma.
\end{equation}
For $\alpha\in \Pi$, the reflection $s_{\alpha}$ is referred to as a simple reflection. We write $\ell$ for the length function on $\W$ with respect to the simple reflections. A simple root $\alpha$ is a right descent of $w\in \W$ if
$\ell(ws_\alpha)<\ell(w)$, and it is a right ascent otherwise. The Bruhat order is denoted by $\leq$, where $u\leq w$ means that a reduced expression for $u$ occurs as a subword of a reduced expression for $w$. We say that $u<w$ is a cover if there does not exist $v$ with $u<v<w$.

For $\gamma \in \Phi^+$ set $\gamma^\vee=2\gamma/\langle\gamma,\gamma\rangle$ for the coroot associated to $\gamma$. Let
$\rho=\frac12\sum_{\gamma\in\Phi^+}\gamma$ be the half sum of the positive roots. The height of a root $\beta=\sum_{\alpha\in\Pi}a_\alpha\alpha$ is
$\operatorname{ht}(\beta)=\sum_\alpha a_\alpha$.

\subsection{Hermitian pairs}\label{sec:Herm} We refer to \cite{EHP}*{Sections 2--3} for background on Hermitian
pairs and their root combinatorics.

Fix a maximal parabolic subalgebra $\p=\m\oplus\mathfrak{u}$ containing
$\mathfrak b$, with Levi factor $\m\supseteq\mathfrak h$ and
nilradical $\mathfrak{u}$. Write $\Phi(\m)$ and $\Phi(\mathfrak{u})$ for their
$\mathfrak h$-roots, so that
\begin{equation}
\mf{m}=\mathfrak{h}\oplus \bigoplus_{\alpha \in \Phi(\m)} \g_{\alpha},\quad\quad \mf{u}=	\bigoplus_{\alpha \in \Phi(\mathfrak{u})}\g_{\alpha},\quad\quad \Phi=\Phi(\mathfrak{u})\sqcup \Phi(\m)\sqcup (-\Phi(\mathfrak{u})).
\end{equation}
We say that roots in $\Phi(\m)$ are compact, while roots in $\Phi\setminus \Phi(\m)$ are noncompact. Since $\p$ is maximal parabolic, there is one noncompact simple root, which we call $\widetilde{\alpha}$:
\begin{equation}
\Phi(\mathfrak{u})\cap \Pi=\{\widetilde{\alpha}\}.    
\end{equation}
We write $\Phi(\m)^+=\Phi(\m)\cap\Phi^+$ for the positive compact roots.

A pair $(\g,\mathfrak{m})$ is called a Hermitian pair if
$\mathfrak{u}$ is abelian. We assume throughout that this is the
case. The seven families are listed in the Introduction (see also \cite{EHP}). By \cite{EHP}*{Lemma 2.2}, abelianness of $\mathfrak{u}$ is equivalent to
$\widetilde{\alpha}$ having coefficient one in the largest root of $(\Phi,\leq)$.

Let $\W_{\m}$ be the Weyl group of $\m$, and let $\W^{\mathfrak{m}}$ denote the set of minimal length coset representatives of $\W_{\m}\backslash \W$. Then $\W=\W_{\m}\W^{\m}$ and each $w\in \W$ can be written uniquely as $w=vx$ with $v\in \W_{\m}$, $x\in \W^{\m}$, and $\ell(w)=\ell(v)+\ell(x)$. In this case, we write $\overline{w}=x$ for the projection from $\W$ to $\W^{\m}$. 

In the following statement, given $\beta\in \Phi$, we write $\tilde{c}(\beta)$ for the coefficient of $\widetilde{\alpha}$ in $\beta$.

\begin{lemma}\label{lem:am-root-conventions}
The following is true about $\Phi(\m)$ and $\Phi(\mathfrak{u})$.
\begin{enumerate}
\renewcommand{\labelenumi}{(\roman{enumi})}
\item The compact and noncompact roots are determined by $\tilde{c}(-)$:
\begin{equation}\label{eq:root-coefficients}
 \tilde{c}(\Phi)\subseteq\{-1,0,1\},\qquad
 \Phi(\m)=\{\beta\in\Phi:\tilde{c}(\beta)=0\},\qquad
 \Phi(\mathfrak{u})=\{\beta\in\Phi:\tilde{c}(\beta)=1\}.
\end{equation}
\item An element $x\in\W$ belongs to $\W^{\m}$ if and only if
$x^{-1}\Phi(\m)^+\subseteq\Phi^+$.
\item The group $\W_{\m}$ preserves $\tilde{c}$ and both sets
$\Phi(\mathfrak{u})$ and $-\Phi(\mathfrak{u})$.
\item For $x\in\W^{\m}$ and $\alpha\in\Pi$,
\[
 xs_\alpha\in\W^{\m}
 \quad\Longleftrightarrow\quad
 x\alpha\notin\Phi(\m).
\]
In this case $\ell(xs_\alpha)=\ell(x)+1$ if and only if
$x\alpha\in\Phi(\mathfrak{u})$.
\end{enumerate}
\end{lemma}
\begin{proof}
For \textup{(i)}, every positive root is bounded above by
the highest root, whose coefficient at $\widetilde{\alpha}$
is one. Thus $\tilde{c}(\Phi^+)\subseteq\{0,1\}$.
The standard parabolic root decomposition identifies
$\Phi(\m)$ with the roots having coefficient zero and
$\Phi(\mathfrak{u})$ with the positive roots having positive
coefficient. This proves \textup{(i)}.

Part \textup{(ii)} is \cite{EHP}*{Remark 3.6(ii)}. For \textup{(iii)}, if
$\alpha\in\Pi\setminus\{\widetilde{\alpha}\}$, then (i) gives
\[
\tilde{c}(s_\alpha\beta)
=\tilde{c}(\beta)-\langle\beta,\alpha^\vee\rangle \tilde{c}(\alpha)
=\tilde{c}(\beta).
\]
These reflections generate $\W_{\m}$, so the assertion
follows from \textup{(i)}.

For \textup{(iv)}, part \textup{(ii)} gives
$x^{-1}\Phi(\m)^+\subseteq\Phi^+$.
Since $s_\alpha$ sends every positive root other than
$\alpha$ to a positive root, we obtain
\[
xs_\alpha\in\W^{\m}
\quad\Longleftrightarrow\quad
\alpha\notin x^{-1}\Phi(\m)^+
\quad\Longleftrightarrow\quad
x\alpha\notin\Phi(\m)^+.
\]
Part \textup{(ii)} also excludes
$x\alpha\in-\Phi(\m)^+$, proving the equivalence.
Finally, $\ell(xs_\alpha)=\ell(x)+1$ if and only if
$x\alpha>0$ \cite{EHP}*{Lemma 3.9}. When $x\alpha$ is
noncompact, this is equivalent to $x\alpha\in\Phi(\mathfrak{u})$.
\end{proof} 

For $w\in\W$, write
\begin{equation}\label{eq:root-inversion-ideal}
 \Phi_w=\Phi^+\cap w(-\Phi^+).
\end{equation}
Then $|\Phi_w|=\ell(w)$ \cite{EHP}*{Section 3.1}.
We define the right weak Bruhat order $\leq_{\mathrm{weak}}$ on $\W$ to be the partial order generated by
$x<_{\mathrm{weak}}xs_\alpha$ for $\alpha\in\Pi$ with
$\ell(xs_\alpha)=\ell(x)+1$. We recall the following.

\begin{lemma}
\label{lem:root-bruhat-ideals}
\begin{enumerate}
\renewcommand{\labelenumi}{(\roman{enumi})}
\item \textup{\cite{EHP}*{Lemma 3.9}.}
For $w\in\W$ and $\alpha\in\Pi$, the reflection
$s_\alpha$ is a right ascent of $w$ if and only if $w\alpha>0$.
When $w\alpha>0$, one has
\begin{equation}\label{eq:root-cover-ideal}
 \Phi_{ws_\alpha}=\Phi_w\cup\{w\alpha\}.
\end{equation}
\item \textup{\cite{EHP}*{Proposition 3.10}.}
For $v,w\in\W$,
\[
 v\leq_{\mathrm{weak}}w
 \quad\Longleftrightarrow\quad\Phi_v\subseteq\Phi_w.
\]
\item \textup{\cite{EHP}*{Corollary 3.12}.}
On $\W^{\m}$, Bruhat order and right weak Bruhat order coincide.
\end{enumerate}
\end{lemma}

\noindent In particular, if $x,y\in\W^{\m}$ and
$\ell(y)=\ell(x)+1$, then
\[
x<y\quad\Longleftrightarrow\quad
y=xs_\alpha\text{ for some }\alpha\in\Pi.
\]

\subsection{Flag and Schubert varieties} Let $G$ be the simply-connected complex linear algebraic group with Lie algebra $\g$, with Borel and parabolic subgroups $B\subseteq P\subseteq G$ corresponding to $\mathfrak{b}\subseteq \mathfrak{p}\subseteq \g$.

Let $X=G/P$ be the corresponding flag variety, of dimension $d_X=\dim(\mathfrak{u})=|\Phi(\mathfrak{u})|$. It has the Bruhat decomposition 
\begin{equation}
X=\bigsqcup_{x\in \W^{\m}} Bw_0x^{-1}P/P,
\end{equation}
where $w_0$ is the longest element of $\W$. The sets $Bw_0x^{-1}P/P$ are the $B$-orbits on $X$, known as the Schubert cells. For $x\in \W^{\m}$ we set
\begin{equation}
O_x=Bw_0x^{-1}P/P,\quad \textnormal{and}\quad Z_x=\overline{O}_x.    
\end{equation}
Here, $Z_x$ is the Schubert variety associated to $x$. We have
\[
Z_y\subseteq Z_x\quad\Longleftrightarrow\quad x\leq y
\qquad(x,y\in\W^{\m}).
\]
The codimension of $Z_x$ is $\ell(x)$.

\subsection{Generalized Verma modules and localization}\label{sec:Verma} The following is a condensed version of \cite{schubertLC}*{Section 3.4}, along with a recollection of structure formulas over Hermitian pairs \cite{cisWeight}.

Let $U(\g)$ be the universal enveloping algebra of $\g$. For $x\in \W^{\m}$ we let $F_x$ denote the finite dimensional irreducible $\mathfrak{m}$-representation of highest weight $x\rho-\rho$. We view $F_x$ as a representation of $\p$, with $\mathfrak{u}$ acting trivially. We consider the parabolic Verma module $M_x$ and its dual
\begin{equation}\label{eq:standards-costandards}
M_x=U(\g)\otimes_{U(\p)}F_x,
\qquad
N_x=M_x^\vee.
\end{equation}
Here $(-)^\vee$ is the BGG dual (see \cite[Section 3.2]{HumphreysO}). The module $M_x$ has a unique simple head $L_x$, the simple module of highest weight $x\rho-\rho$, which is the simple socle of $N_x$.

The modules $M_x$, $N_x$, $L_x$ have geometric realizations, via the Beilinson--Bernstein localization theorem \cite{BB}. More precisely (see \cite{schubertLC}*{Section 3.4}),
\begin{equation}
N_x=\Gamma\big(X,\mathscr{H}^{\ell(x)}_{O_x}(\mathscr{O}_X)\big), \quad M_x=\Gamma\big(X,\mathbb{D}\mathscr{H}^{\ell(x)}_{O_x}(\mathscr{O}_X)\big),\quad L_x=\Gamma(X,\operatorname{IC}_{Z_x}),
\end{equation}
where $\mathbb{D}$ is the duality functor of $\mathscr{D}_X$-modules. As such, $N_x$ and $M_x$ are endowed with weight filtrations $W_{\bullet}$, via the global sections functor, if we endow $\mathscr{H}^{\ell(x)}_{O_x}(\mathscr{O}_X)$ with the Hodge structure $\mathscr{H}^{\ell(x)}_{O_x}(\mathscr{O}_X^H)$:
\begin{equation}\label{eq:geometric-costandard}
W_p(N_x)=   \Gamma\big(X,W_p(\mathscr{H}^{\ell(x)}_{O_x}(\mathscr{O}_X^H))\big),\quad  p\in \mathbb{Z}.
\end{equation}
We recall the weight filtration on $N_x$ in the case of Hermitian pairs. Following \cite{cisWeight}, we make the following definitions. Set
\begin{equation}\label{eq:root-incompatibility}
 \mathscr M=\{(\gamma,\nu) \mid \gamma,\nu\in\Phi^+,\;
 \langle\gamma,\nu\rangle\ne0
 \text{ or both roots are long}\},
\end{equation}
with the convention that, if only one root length occurs, all roots are short. We denote by $\mathscr{S}$ the set of all subsets $\Omega$ of $\Phi^+$ that satisfy the following two conditions:
\begin{enumerate}
\item [(a)]  If $\gamma$ and $\nu$ are both in $\Omega$ and $\gamma \neq \nu$, then $(\gamma,\nu)\notin \mathscr{M}$. 

\item [(b)] If $\gamma$ is in $\Omega$ and there is $\xi$ in $\Phi^+$ such that $\gamma \neq \xi$, $(\gamma, \xi)\in \mathscr{M}$, and $\xi \leq \gamma$, then there is $\zeta \in \Omega$ satisfying $\zeta \neq \gamma$, $(\zeta,\xi)\in \mathscr{M}$, and $\zeta\leq \gamma$.
\end{enumerate}

For $x\in \W^{\m}$ we define
\begin{equation}\label{eq:admissible-noncompact}
\mathscr{S}_x=\{ \Omega \in \mathscr{S} \mid x\Omega \subseteq \Phi(\mathfrak{u})\cup -\Phi(\mathfrak{u})\}.
\end{equation}
We write $\mathscr{E}_x$ for the set of $\Omega$ in
$\mathscr{S}_x$ from \eqref{eq:admissible-noncompact}
satisfying the additional condition:
\begin{enumerate}
\item [(c)] If $\gamma$ is in $\Omega$ then there is $\zeta$ in $\Omega$ with $\gamma \leq \zeta$ and $x\zeta \in \Phi(\mathfrak{u})$.   
\end{enumerate}
For any $\Omega\subseteq\Phi^+$ and $x\in\W$, set
\begin{equation}\label{eq:admissible-positive}
\Omega_x^+
=
\{\gamma\in\Omega\mid x\gamma\in\Phi(\mathfrak{u})\}.
\end{equation}
If $\Omega$ is pairwise orthogonal, define
\begin{equation}\label{eq:admissible-label}
r_{x,\Omega}
=
\prod_{\gamma\in\Omega_x^+}s_\gamma,
\qquad
t_x(\Omega)=\overline{x r_{x,\Omega}}.
\end{equation}
As the reflections $s_{\gamma}$ for $\gamma\in \Omega_x^+$ commute, the product $r_{x,\Omega}$ is independent of its
ordering.

By \cite{intertwining}*{Propositions 2.3 and 4.3}, the map
\begin{equation}\label{eq:admissible-factors}
\mathscr E_x\longrightarrow\W^{\m},
\qquad
\Omega\longmapsto t_x(\Omega),
\end{equation}
is injective, and its image indexes the composition factors of
$M_x$ and $N_x$, each occurring with multiplicity one. Furthermore, the size of $\Omega_x^+$ controls the weight level of a composition factor.

\begin{prop}\label{prop:am-weight-normalization}
For all $p\in \mathbb{Z}$ we have that 
\begin{equation}\label{eq:am-weight-layers}
\operatorname{Gr}^{W}_{d_X+\ell(x)+p}N_x
\cong
\bigoplus_{\substack{\Omega\in\mathscr E_x\\
                    |\Omega_x^+|=p}}
L_{t_x(\Omega)}.    
\end{equation}
\end{prop}
\begin{proof}
The composition factors of $M_x$ (and hence $N_x$) are described in \cite[Theorem 1.3]{cisWeight}. The authors also show that the socle, radical, and \textit{$\ell$-adic weight} filtrations coincide and are the unique Loewy filtration \cite[Corollary 1.7]{cisWeight}.  The graded pieces of the Loewy filtration are interpreted in terms of inverse parabolic Kazhdan--Lusztig polynomials in \cite[Corollary 7.1.3]{irvingBook}.  In \cite[Theorem 1.4]{KT} it is shown that these polynomials also describe the weight filtration on $N_x$. We have normalized the shift on $W_{\bullet}$ so that $L_x$ lives in weight $p=d_X+\ell(x)$ of $N_x$, consistent with (\ref{eq:geometric-costandard}).
\end{proof}

Every polarizable pure Hodge module with underlying $\mathscr D_X$-module
$\operatorname{IC}_{Z_t}$ is a Tate twist of
$\operatorname{IC}_{Z_t}^H$ (see, for instance, \cite{PRdet}*{Section 2.3}).
Thus a summand with label $t$ in weight $p$ is
\[
\operatorname{IC}_{Z_t}^H
\left(\frac{d_X-\ell(t)-p}{2}\right).
\]
In particular, $d_X-\ell(t)-p$ is even whenever this summand occurs.

\subsection{Local cohomology and the Grothendieck--Cousin complex}
\label{subsec:am-geometric-conventions}

For $x\in\W^{\m}$, \cite{schubertLC}*{Theorem 3.1} gives a
Grothendieck--Cousin complex $\mathscr{GC}_x^{\bullet}$ of mixed Hodge modules with
\begin{equation}\label{eq:lc-cousin-terms}
\mathscr{GC}_x^q
=\bigoplus_{\substack{y\in\W^{\m}\\ x\leq y,\ \ell(y)=q}}
\mathscr{H}^{\ell(y)}_{O_y}(\mathscr{O}_X^H),
\qquad
H^q(\mathscr{GC}_x^\bullet)
\cong\mathscr H^q_{Z_x}(\mathscr O_X^H).
\end{equation}
Importantly, for $y,z\in \W^{\m}$ with $x\leq y$, $x\leq z$, and $\ell(z)=\ell(y)+1$, the map
\begin{equation}
\tilde{\delta}_{y,z}: \mathscr{H}^{\ell(y)}_{O_y}(\mathscr{O}_X^H) \to \mathscr{H}^{\ell(z)}_{O_z}(\mathscr{O}_X^H),   
\end{equation}
induced by the differential in $\mathscr{GC}_x^{\bullet}$, is nonzero if and only if $y\leq z$. By the results in Section \ref{sec:Herm}, this is further equivalent to $z=ys_{\alpha}$ for some $\alpha\in \Pi$ with $y\alpha\in \Phi(\mathfrak{u})$.

Taking global sections gives a complex $(GC^{\bullet}_{x},\partial)$ of $U(\g)$-modules with
\begin{equation}
GC_x^q
=\bigoplus_{\substack{y\in\W^{\m}\\ x\leq y,\ \ell(y)=q}} N_y,\quad \textnormal{and}\quad\partial|_{N_y}=\sum_{\substack{\alpha\in \Pi\,\; ys_{\alpha}\in \W^{\m}\\ \ell(ys_{\alpha})=\ell(y)+1}} \delta_{y,ys_{\alpha}},   
\end{equation}
where $\delta=\Gamma(X,\tilde{\delta})$. These maps are strict with respect to the weight filtration:
\begin{equation}\label{eq:lc-strictness}
\operatorname{Gr}^W_p H^q(GC^{\bullet}_{x})
\cong
H^q\!\left(
  \operatorname{Gr}^W_p GC_x^\bullet
\right).
\end{equation}
Furthermore, we have that $\operatorname{Hom}_{U(\g)}(N_y,N_{ys_{\alpha}})$ is one dimensional \cite{intertwining}*{Proposition 2.2}. Thus, the map $\delta_{y,ys_{\alpha}}$ is unique up to non-zero scalar.

\subsection{Translation functors}\label{subsec:am-translation}

Let $\mathcal O$ denote the BGG category of $\g$ \cite{HumphreysO}. For an integral
weight $\eta$, write $\mathcal O_\eta$ for the block with the
central character of highest weight $\eta$, and
$\operatorname{pr}_\eta$ for projection to this block from $\mc{O}$.
Let $\Delta(\eta)$ be the ordinary Verma module of highest
weight $\eta$, with simple quotient $L(\eta)$.
We use the dot action $w\cdot\eta=w(\eta+\rho)-\rho$ and
extend the notation $L_w=L(w\cdot0)$ to all $w\in\W$.

We consider the translation functors of \cite[Chapter 7]{HumphreysO}. Fix $\alpha\in\Pi$, put $s=s_\alpha$, and set
$\mu=-\omega_\alpha$, where $\omega_\alpha$ is the corresponding
fundamental weight. Then $\mu+\rho$ is dominant with stabilizer
$\{e,s\}$. Translation onto and out of the $s$-wall gives functors
\begin{equation}
T=T_0^\mu:\mathcal O_0\longrightarrow\mathcal O_\mu,
\qquad
T'=T_\mu^0:\mathcal O_\mu\longrightarrow\mathcal O_0.
\end{equation}
Explicitly, $T(V)=\operatorname{pr}_\mu(L(\omega_\alpha)^*\otimes V)$.
These functors are exact, and $T'$ is both a left and a right
adjoint of $T$ \cite{HumphreysO}*{Section 7.1 and Proposition 7.2}.
For every $w\in\W$, \cite{HumphreysO}*{Theorems 7.6 and 7.9} give
\begin{equation}\label{eq:am-wall-translation}
 T\Delta(w\cdot0)\cong\Delta(w\cdot\mu),\qquad
 TL_w\cong
 \begin{cases}
 L(w\cdot\mu),&ws<w,\\
 0,&ws>w.
 \end{cases}
\end{equation}

Let $\mathcal O_\eta^{\p}$ be the full subcategory of
$\p$-locally finite objects in $\mathcal O_\eta$, and let
$Q_\eta:\mathcal O_\eta\to\mathcal O_\eta^{\p}$ take the maximal
$\p$-locally finite quotient \cite[Chapter 9]{HumphreysO}. Then $Q_\eta$ is left adjoint to
the inclusion, and $Q_0\Delta(w\cdot0)=M_w$ for $w\in\W^{\m}$
\cite{HumphreysO}*{Section 9.3 and Theorem 9.4(c)}.
Both $T$ and $T'$ preserve $\p$-local finiteness
\cite{HumphreysO}*{Proposition 9.3(d)}. Thus $TQ_0$ and
$Q_\mu T$ are left adjoint to $T':\mathcal O_\mu^{\p}\to\mathcal O_0$.
Uniqueness of left adjoints and \eqref{eq:am-wall-translation} give
\begin{equation}\label{eq:am-parabolic-translation}
 TQ_0\cong Q_\mu T,\qquad
 TM_w\cong Q_\mu\Delta(w\cdot\mu)
 \quad(w\in\W^{\m}).
\end{equation}
We will use these facts to prove Theorem \ref{thm:adjacentmaps} on maps of parabolic Verma modules.

\section{Adjacent standard maps}\label{sec:adjacent-proof}

The goal of this section is to determine the images of the differentials in the Grothendieck--Cousin complexes (see Section \ref{subsec:am-geometric-conventions}). The main result is Theorem \ref{thm:adjacentmaps}, which describes images of maps of dual Verma modules in the case of a right ascent.

\subsection{Images of adjacent maps}

Let $x\in\W^{\m}$ and $\alpha\in\Pi$ satisfy
\[
y=xs_\alpha\in\W^{\m},
\qquad
\ell(y)=\ell(x)+1.
\]
By \cite{intertwining}*{Proposition 2.2}, the space
$\operatorname{Hom}_{U(\g)}(N_x,N_y)$ is one-dimensional. Consider a (unique up to nonzero scalar) nonzero map
\begin{equation}\label{eq:adjacent-dual-map}
\delta_{x,y}:N_x\longrightarrow N_y.
\end{equation}
By the discussion of Section \ref{subsec:am-geometric-conventions}, this map is strict with respect to the weight filtration:
\[
\delta_{x,y}(W_pN_x)
=
\operatorname{im}(\delta_{x,y})\cap W_pN_y
\qquad(p\in\mathbb Z).
\]
Using notation as in Section \ref{sec:Verma}, we describe the image of $\delta_{x,y}$.

\begin{theorem}\label{thm:adjacentmaps}
Let $x\in\W^{\m}$ and $\alpha\in\Pi$ satisfy $y=xs_\alpha\in\W^{\m}$ and $\ell(y)=\ell(x)+1$. For every $p\in\mathbb Z$, there is an isomorphism of
$U(\g)$-modules
\begin{equation}\label{eq:adjacent-weight-image}
\operatorname{im}\!\left(
  \operatorname{Gr}^{W}_{p}\delta_{x,y}
\right)
\cong
\bigoplus_{\substack{
  \Omega\in\mathscr E_x\\
  d_X+\ell(x)+|\Omega_x^+|=p\\
  \alpha\in\Omega
}}
L_{t_x(\Omega)}.
\end{equation}
\end{theorem}

\noindent In the process of the proof, we will show that this is a ``maximal rank" statement, in the sense that
\begin{equation}\label{eq:maxRank}
[\operatorname{im}(\delta_{x,y}): L_t]=[N_x:L_t][N_y:L_t],\quad \textnormal{for $t\in \W^{\m}$},   
\end{equation}
recalling that $N_x$ and $N_y$ have multiplicity-free composition series.

\subsection{Common factors of standard modules} In this subsection, we take steps towards determining $\operatorname{im}(\delta_{x,y})$ by describing the common composition factors of $N_x$ and $N_y$. The main result is:

\begin{prop}\label{prop:am-common}
Let $x\in\W^{\m}$ and $\alpha\in\Pi$ satisfy
$y=xs_\alpha\in\W^{\m}$ and $\ell(y)=\ell(x)+1$.
For $\Omega\in\mathscr E_x$:
\begin{equation}\label{eq:am-common}
[N_y:L_{t_x(\Omega)}]=1\quad\Longleftrightarrow\quad\alpha\in\Omega.
\end{equation}
When these conditions hold
\begin{equation}
[\operatorname{Gr}^{W}_{d_X+\ell(x)+|\Omega_x^+|}N_y:L_{t_x(\Omega)}]=1.    
\end{equation}
\end{prop}

\noindent Thus, in this setting, $L_{t_x(\Omega)}$ appears in the same weight level of $N_x$ and $N_y$. We begin with a lemma.

\begin{lemma}\label{lem:comp}
Let $t\in\W^{\m}$. If $\beta,\gamma\in \Phi^+$ satisfy $\beta\neq \gamma$, $t\beta,t\gamma\in-\Phi(\mathfrak{u})$, and
$(\beta,\gamma)\in\mathscr M$, then $\beta$ and $\gamma$ are
comparable in the root order, i.e. $\beta<\gamma$ or $\gamma<\beta$. 
\end{lemma}

\begin{proof}
We split the argument into three cases, depending on whether
$\langle\beta,\gamma\rangle$ is negative, positive, or zero. If $\langle\beta,\gamma\rangle<0$ then
$\beta+\gamma$ is a root, by the root-string property. But $t(\beta+\gamma)\in\Phi$ would satisfy
\[
 \tilde{c}(t(\beta+\gamma))=\tilde{c}(t\beta)+\tilde{c}(t\gamma)=-2,
\]
contradicting \eqref{eq:root-coefficients}. Thus this case cannot occur. If $\langle\beta,\gamma\rangle>0$, the root-string property gives
$\beta-\gamma\in\Phi$. If this difference is positive then
$\gamma<\beta$; if it is negative then $\beta<\gamma$.

Finally, suppose $\langle\beta,\gamma\rangle=0$.
By \eqref{eq:root-incompatibility}, the hypothesis
$(\beta,\gamma)\in\mathscr M$ forces both roots to be long.
Under our convention that all roots in a simply laced system are
short, this leaves only types $\mathsf B$ and $\mathsf C$.
For these two Hermitian pairs, use standard orthonormal
coordinates $\epsilon_i$. Put $u=-t\beta$ and
$v=-t\gamma$, which by hypothesis are orthogonal long roots in $\Phi(\mathfrak{u})$.
In type $\mathsf{B}$, the only such pair is
$\{\epsilon_1-\epsilon_j,\epsilon_1+\epsilon_j\}$ for a common
$j>1$, whose difference is $\pm2\epsilon_j$.
In type $\mathsf{C}$, the pair must be $\{2\epsilon_i,2\epsilon_j\}$ with $i\ne j$,
whose difference is $2(\epsilon_i-\epsilon_j)$.
Thus in either case $u-v=2\eta$ for some root $\eta$, and $\beta-\gamma=-2t^{-1}\eta$.
Since $t^{-1}\eta$ is a root, its simple-root coefficients are
either all nonnegative or all nonpositive. Therefore, the same is true of
$\beta-\gamma$, proving the desired comparability.
\end{proof}

For $t\in\W^{\m}$, we write
\begin{equation}\label{eq:negative-admissible-sets}
 \mathscr S_t^-=
 \{\Psi\in\mathscr S \mid t\Psi\subset-\Phi(\mathfrak{u})\}.
\end{equation}
We first describe the sets in \eqref{eq:negative-admissible-sets}.
This will allow us to determine all common composition factors of
two adjacent standards. We refer freely to conditions (a) and (b) of Section \ref{sec:Verma}.

\begin{lemma}\label{lem:am-negative-union}
The union of any two members of $\mathscr S_t^-$ is a member of
$\mathscr S_t^-$. Consequently this collection has a greatest member,
denoted $\mathscr{R}_t$.
\end{lemma}
\begin{proof}

Let $\Psi,\Lambda\in\mathscr S_t^-$. We first show that
$\Psi\cup\Lambda\in\mathscr S$ by verifying conditions (a) and (b).

If $\Psi\cup\Lambda$ fails condition (a), choose distinct roots $\gamma\in\Psi$,
$\xi\in\Lambda$ with $(\gamma,\xi)\in\mathscr M$ such that
$\operatorname{ht}(\gamma)+\operatorname{ht}(\xi)$ is minimal.
Interchanging $\Psi$ and $\Lambda$ if necessary,
Lemma \ref{lem:comp} gives $\xi<\gamma$. Condition (b) for $\Psi$
provides $\zeta\in\Psi\setminus\{\gamma\}$ with
$\zeta\leq\gamma$ and $(\zeta,\xi)\in\mathscr M$.
If $\zeta=\xi$, condition (a) in $\Psi$ contradicts the
choice of $\gamma,\xi$. Thus, $\zeta<\gamma$, which implies
$\operatorname{ht}(\zeta)+\operatorname{ht}(\xi)
<\operatorname{ht}(\gamma)+\operatorname{ht}(\xi)$.
Since $\zeta\in\Psi$, $\xi\in\Lambda$, $\zeta\ne\xi$, and
$(\zeta,\xi)\in\mathscr M$, this contradicts the minimality of
$\operatorname{ht}(\gamma)+\operatorname{ht}(\xi)$. This proves (a).
For (b), take $\gamma\in\Psi\cup\Lambda$ and
$\xi\in\Phi^+$ with $\xi<\gamma$ and $(\gamma,\xi)\in\mathscr M$.
Apply (b) in whichever of $\Psi,\Lambda$ contains $\gamma$, so there exists $\zeta$ in $\Psi\cup\Lambda$ such that $\zeta\neq \gamma$, $(\zeta,\xi)\in \mathscr{M}$, and $\zeta\leq \gamma$.
Thus $\Psi\cup\Lambda$ satisfies (b), so belongs to $\mathscr S$.

Finally, $t(\Psi\cup\Lambda)=t\Psi\cup t\Lambda\subseteq-\Phi(\mathfrak{u})$.
Thus $\Psi\cup\Lambda\in\mathscr S_t^-$.
The collection $\mathscr S_t^-$ is finite and contains the empty
set, so repeated unions give
$\mathscr{R}_t=\bigcup_{\Psi\in\mathscr S_t^-}\Psi\in\mathscr S_t^-$,
which contains every member.
\end{proof}

Since $\mathscr{R}_t\in\mathscr S$, condition (a) makes its roots
pairwise orthogonal. Their reflections therefore commute, so
the product below is independent of its ordering.
For $J\subseteq\mathscr{R}_t$, define
\begin{equation}\label{eq:am-column}
 x_J=\overline{t\prod_{\gamma\in J}s_\gamma},\qquad
 \Omega_J=\{\eta\in\mathscr{R}_t \mid \eta\leq\gamma
                         \text{ for some }\gamma\in J\}.
\end{equation}

The following result is proven in the classical cases in 
\cite{EnrightShelton}*{Theorem 8.4(ii)}. In their notation, our $\mathscr{R}_t$ is equal to the set $-t^{-1}\Sigma_t$. We give a uniform argument for all Hermitian pairs.

\begin{lemma}\label{lem:am-column}
The dual Verma modules having $L_t$ as a composition factor
are precisely the $N_{x_J}$ with $J\subseteq\mathscr{R}_t$.
Distinct subsets $J$ give distinct indices $x_J$. For each such $J$,
\begin{equation}\label{eq:J}
 \Omega_J\in\mathscr E_{x_J},\qquad
 (\Omega_J)^+_{x_J}=J,\qquad
 t_{x_J}(\Omega_J)=t.
\end{equation}
\end{lemma}
\begin{proof}
Fix $J\subseteq\mathscr{R}_t$. We first prove the inclusions
\begin{equation}\label{eq:am-column-signs}
 x_JJ\subseteq\Phi(\mathfrak{u}),\qquad
 x_J(\mathscr{R}_t\setminus J)\subseteq-\Phi(\mathfrak{u}).
\end{equation}
Put $R_J=\prod_{\gamma\in J}s_\gamma$ and write
$tR_J=v_Jx_J$ with $v_J\in\W_{\m}$.
Fix $\eta\in\mathscr{R}_t$. For $\gamma\in J$ with $\gamma\ne\eta$,
orthogonality gives $\langle\eta,\gamma\rangle=0$, so
$s_\gamma\eta=\eta$.
If $\eta\in J$, the factor $s_\eta$ in the product $R_J$ sends
$\eta$ to $-\eta$, and every other factor of $R_J$ fixes both
$\eta$ and $-\eta$.
If $\eta\notin J$, every factor of $R_J$ fixes $\eta$. Hence
\[
 R_J\eta=\begin{cases}-\eta,&\eta\in J,\\\eta,&\eta\in\mathscr{R}_t\setminus J.\end{cases}
\]
Also, $tR_J=v_Jx_J$ implies $x_J=v_J^{-1}tR_J$, and therefore
$x_J\eta=v_J^{-1}tR_J\eta$.
Since $t\eta\in-\Phi(\mathfrak{u})$ and $v_J^{-1}$ preserves both
$\Phi(\mathfrak{u})$ and $-\Phi(\mathfrak{u})$ by Lemma~\ref{lem:am-root-conventions}\textup{(iii)},
$x_J\eta\in\Phi(\mathfrak{u})$ exactly for $\eta\in J$.

Next, we prove (\ref{eq:J}). Since $\Omega_J\subseteq\mathscr{R}_t$, \eqref{eq:am-column-signs} gives
$x_J\Omega_J\subseteq\Phi(\mathfrak{u})\cup-\Phi(\mathfrak{u})$.
The set $\Omega_J$ satisfies (a) since $\mathscr{R}_t$ does.
For (b), take $\eta\in\Omega_J$ and $\xi<\eta$ with
$(\eta,\xi)\in\mathscr M$. Condition (b) in $\mathscr{R}_t$ supplies
$\zeta\in\mathscr{R}_t\setminus\{\eta\}$ such that
$\zeta\leq\eta$ and $(\zeta,\xi)\in\mathscr M$. At the same time, since $\eta\in\Omega_J$, the definition
of $\Omega_J$ gives $\gamma\in J$ with $\eta\leq\gamma$.
Because $\zeta\leq\eta\leq\gamma$, it follows that $\zeta$ belongs to
$\Omega_J$.
For (c), let $\eta\in\Omega_J$. By definition, there is
$\gamma\in J$ with $\eta\leq\gamma$.
Since $J\subseteq\mathscr{R}_t$ and $\gamma\leq\gamma$, the definition
of $\Omega_J$ also gives $\gamma\in\Omega_J$.
The inclusion $x_JJ\subseteq\Phi(\mathfrak{u})$ in \eqref{eq:am-column-signs} gives
$x_J\gamma\in\Phi(\mathfrak{u})$. Thus $\gamma$ is the root required
by condition (c) for $\eta$, proving $\Omega_J\in\mathscr E_{x_J}$.
Since $J\subseteq\Omega_J\subseteq\mathscr{R}_t$ by definition,
\eqref{eq:am-column-signs} also gives $(\Omega_J)_{x_J}^+=J$.
Since $R_J^2=e$,
\[
 t_{x_J}(\Omega_J)=\overline{x_JR_J}
 =\overline{v_J^{-1}t}=t.
\]
Proposition \ref{prop:am-weight-normalization} gives $[N_{x_J}:L_t]=1$.

Conversely, suppose $[N_x:L_t]=1$. Choose the unique
$\Omega\in\mathscr E_x$ with $t_x(\Omega)=t$, and put
$J=\Omega_x^+$ and $R_J=\prod_{\gamma\in J}s_\gamma$.
By definition of $t_x(\Omega)$, there exists $v\in\W_{\m}$ such that $xR_J=vt$.
Since $\Omega\in\mathscr E_x$ is pairwise orthogonal and
$J=\Omega_x^+$, we have
\[
xR_J\eta=
\begin{cases}
-x\eta,&\eta\in J,\\
 x\eta,&\eta\in\Omega\setminus J,
\end{cases}
\]
which, in either case, belongs to $-\Phi(\mathfrak{u})$. Applying $v^{-1}$, which preserves $-\Phi(\mathfrak{u})$ by
Lemma~\ref{lem:am-root-conventions}\textup{(iii)}, gives
$t\Omega=v^{-1}xR_J\Omega\subseteq-\Phi(\mathfrak{u})$.
Therefore $\Omega\in\mathscr S_t^-$ and $\Omega\subseteq\mathscr{R}_t$.
Moreover $tR_J=v^{-1}x$, so
$x_J=\overline{tR_J}=x$.

The map $J\mapsto x_J$ is injective because
\eqref{eq:am-column-signs} recovers
$J=\{\gamma\in\mathscr{R}_t:x_J\gamma\in\Phi(\mathfrak{u})\}$ from $x_J$.
\end{proof}

Lemma \ref{lem:am-column} gives the following descent property.

\begin{lemma}\label{lem:am-common-descent}
Let $x,t\in\W^{\m}$ and $\alpha\in\Pi$ satisfy
$y=xs_\alpha\in\W^{\m}$ and $\ell(y)=\ell(x)+1$.
If $L_t$ is a composition factor of both $N_x$ and $N_y$, then $t\alpha\in-\Phi(\mathfrak{u})$.
In particular, $ts_\alpha<t$.
\end{lemma}

\begin{proof}
By Lemma~\ref{lem:am-column}, write $x=x_J$ and $y=x_K$
with $J,K\subseteq\mathscr R_t$.
Choose $\varpi\in\mathbb R\Phi$ such that
$\langle\varpi,\beta\rangle=\tilde c(\beta)$ for every
$\beta\in\Phi$. For $\gamma\in\mathscr R_t$, we have $t\gamma\in-\Phi(\mathfrak{u})$, so
\begin{equation}
 \langle t^{-1}\varpi,\gamma\rangle
 =\langle\varpi,t\gamma\rangle
 =\tilde c(t\gamma)=-1.
\end{equation}
By \eqref{eq:root-reflection} this gives
$s_\gamma(t^{-1}\varpi)=t^{-1}\varpi+\gamma^\vee$.
Since the roots in $\mathscr R_t$ are pairwise orthogonal,
$s_\gamma$ fixes $\xi^\vee$ for every
$\xi\in\mathscr R_t\setminus\{\gamma\}$.
Moreover, $\W_{\m}$ fixes $\varpi$ by
Lemma~\ref{lem:am-root-conventions}.
Thus \eqref{eq:am-column} gives
\begin{equation}\label{eq:impLem}
 x^{-1}\varpi
 =\left(\prod_{\gamma\in J}s_\gamma\right)t^{-1}\varpi
 =t^{-1}\varpi+\sum_{\gamma\in J}\gamma^\vee,\quad \textnormal{and, similarly}\quad
 y^{-1}\varpi
 =t^{-1}\varpi+\sum_{\gamma\in K}\gamma^\vee.
\end{equation}
On the other hand, Lemma~\ref{lem:am-root-conventions}\textup{(iv)}
gives $x\alpha\in\Phi(\mathfrak{u})$, so
$\langle x^{-1}\varpi,\alpha\rangle=\tilde c(x\alpha)=1$.
Since $y=xs_\alpha$, \eqref{eq:root-reflection} gives
\[
 y^{-1}\varpi=s_\alpha(x^{-1}\varpi)
 =x^{-1}\varpi-\alpha^\vee.
\]
Subtracting the formulas (\ref{eq:impLem}) for $x^{-1}\varpi$ and
$y^{-1}\varpi$, we obtain
\begin{equation}\label{eq:alphaVee}
 \alpha^\vee
 =\sum_{\gamma\in J\setminus K}\gamma^\vee
  -\sum_{\gamma\in K\setminus J}\gamma^\vee.
\end{equation}
Since $\mathscr{R}_t$ is pairwise orthogonal, we get
\begin{equation}\label{eq:am-coroot-squares}
 \langle\alpha^\vee,\alpha^\vee\rangle=\sum_{\gamma\in J\setminus K}
       \langle\gamma^\vee,\gamma^\vee\rangle+\sum_{\gamma\in K\setminus J}
       \langle\gamma^\vee,\gamma^\vee\rangle.
\end{equation}
We claim that $J\setminus K=\{\alpha\}$ and $K\setminus J=\varnothing$. To see this, assume for contradiction that two or more roots occur in sums (\ref{eq:am-coroot-squares}). Since $J,K\subseteq \mathscr{R}_t$, condition (a) and
\eqref{eq:root-incompatibility} forces the existence of a short root $\gamma$ in $J\setminus K$ or $K\setminus J$. The root $\gamma$ contributes
$\langle\gamma^\vee,\gamma^\vee\rangle$ to the right-hand
side of \eqref{eq:am-coroot-squares}, and at least one
other root contributes a strictly positive term, so $\langle\alpha^\vee,\alpha^\vee\rangle
 >\langle\gamma^\vee,\gamma^\vee\rangle$ .
Since $\gamma$ is short, we obtain
\[
 \langle\alpha^\vee,\alpha^\vee\rangle
 >\langle\gamma^\vee,\gamma^\vee\rangle
   =\frac{4}{\langle\gamma,\gamma\rangle}\geq \frac{4}{\langle\alpha,\alpha\rangle}=\langle\alpha^\vee,\alpha^\vee\rangle,
\]
a contradiction. Thus $J$ and $K$ differ by exactly one root. This root cannot belong to $K\setminus J$, since the formula
(\ref{eq:alphaVee}) would then equate a positive coroot with
a negative coroot. It therefore belongs to $J\setminus K$,
and its coroot equals $\alpha^\vee$. Therefore, $J\setminus K=\{\alpha\}$ and $K\setminus J=\varnothing$, as claimed. In particular, $\alpha\in\mathscr R_t$, so
$t\alpha\in-\Phi(\mathfrak{u})$ and $ts_\alpha<t$.
\end{proof}

\begin{proof}[Proof of Proposition~\ref{prop:am-common}]
Put $t=t_x(\Omega)$ and $J=\Omega_x^+$.
Lemma~\ref{lem:am-column} gives
$J\subseteq\mathscr R_t$ and $x=x_J$.
By Lemma~\ref{lem:am-root-conventions}\textup{(iv)} we have $x\alpha\in\Phi(\mathfrak{u})$. Thus,
\eqref{eq:admissible-positive} gives
\begin{equation}\label{eq:am-common-positive-root}
 \alpha\in\Omega\quad\Longleftrightarrow\quad\alpha\in J.
\end{equation}

Suppose first that $[N_y:L_t]=1$.
Lemma~\ref{lem:am-common-descent} gives
$t\alpha\in-\Phi(\mathfrak{u})$, so that $\{\alpha\}\in\mathscr S_t^-$,
and Lemma~\ref{lem:am-negative-union} gives
$\alpha\in\mathscr R_t$.
Now \eqref{eq:am-column-signs} and $x\alpha\in\Phi(\mathfrak{u})$
imply $\alpha\in J\subseteq\Omega$.

Conversely, suppose $\alpha\in\Omega$.
By \eqref{eq:am-common-positive-root}, we have $\alpha\in J$.
Using \eqref{eq:am-column}, write
$t\prod_{\gamma\in J}s_\gamma=vx$ with $v\in\W_{\m}$.
The reflections commute and $s_\alpha^2=e$, so
\[
 t\prod_{\gamma\in J\setminus\{\alpha\}}s_\gamma
 =\left(t\prod_{\gamma\in J}s_\gamma\right)s_\alpha
 =vxs_\alpha=vy.
\]
Since $y\in\W^{\m}$, \eqref{eq:am-column} gives
$x_{J\setminus\{\alpha\}}=y$.
Applying Lemma~\ref{lem:am-column} to $J\setminus\{\alpha\}$,
we obtain
\[
 t_y(\Omega_{J\setminus\{\alpha\}})=t,\qquad
 (\Omega_{J\setminus\{\alpha\}})^+_y=J\setminus\{\alpha\}.
\]
Since $\ell(y)=\ell(x)+1$ and
$|J\setminus\{\alpha\}|=|\Omega_x^+|-1$,
\eqref{eq:am-weight-layers} gives
\[
 [\operatorname{Gr}^{W}_{d_X+\ell(x)+|\Omega_x^+|}
   N_y:L_t]=1,
\]
as required to complete the argument.
\end{proof}

\subsection{Completion of the proof} For convenience in this section, we work with
\begin{equation}
f_{y,x}=\delta_{x,y}^{\vee}:M_y\longrightarrow M_x.    
\end{equation}
We first relate Theorem \ref{thm:adjacentmaps} to Proposition \ref{prop:am-common} by reducing to the assertion (\ref{eq:maxRank}).

\begin{lemma}\label{lem:am-weight-reduction}
Let $x\in\W^{\m}$ and $\alpha\in\Pi$ satisfy $y=xs_\alpha\in\W^{\m}$ and $\ell(y)=\ell(x)+1$. To prove Theorem~\ref{thm:adjacentmaps},
it suffices to prove that 
\begin{equation}\label{eq:fmap}
[\operatorname{im}(f_{y,x}): L_t]=[M_y:L_t][M_x:L_t],
\end{equation}
for all $t\in \W^{\m}$.
\end{lemma}

\begin{proof}
Suppose that (\ref{eq:fmap}) holds. Exactness of BGG duality gives that \eqref{eq:maxRank} holds. 

Let $t\in \W^{\m}$ be such that $L_t$ is a summand of $\operatorname{im}(
  \operatorname{Gr}^{W}_{p}\delta_{x,y}
)$. Then by strictness of $\delta_{x,y}$, $L_t$ is a summand of $
  \operatorname{Gr}^{W}_{p}N_x
$ and $
  \operatorname{Gr}^{W}_{p}N_y
$. By Propositions \ref{prop:am-weight-normalization} and \ref{prop:am-common}, we have that $t=t_x(\Omega)$ for some $\Omega\in \mathscr{E}_x$ with $\alpha\in \Omega$ and $d_X+\ell(x)+|\Omega^+_x|=p$. Thus, $L_t$ is a summand of the right side of (\ref{eq:adjacent-weight-image}).

Conversely, if $t=t_x(\Omega)$ for some $\Omega\in \mathscr{E}_x$ with $\alpha\in \Omega$ and $d_X+\ell(x)+|\Omega^+_x|=p$, then Propositions \ref{prop:am-weight-normalization} and \ref{prop:am-common} give that 
\[
[\operatorname{Gr}^{W}_{p}N_x:L_t][\operatorname{Gr}^{W}_{p}N_y:L_t]=1.
\]
Therefore, \eqref{eq:maxRank} and strictness give that $L_t$ is a summand of $\operatorname{im}(
  \operatorname{Gr}^{W}_{p}\delta_{x,y})$.
\end{proof}

\begin{proof}[Proof of Theorem~\ref{thm:adjacentmaps}]
Write $s=s_\alpha$ and $y=xs>x$, and let $\mu=-\omega_\alpha$
and $T=T_0^\mu$ be as in Section~\ref{subsec:am-translation}.
Since $s\cdot\mu=\mu$, we have $x\cdot \mu=y\cdot \mu$, so \eqref{eq:am-parabolic-translation} gives
\begin{equation}\label{eq:transIsom}
 TM_x\cong Q_\mu\Delta(x\cdot\mu)
 =Q_\mu\Delta(y\cdot\mu)\cong TM_y.
\end{equation}
Put $f=f_{y,x}$ and $I=\operatorname{im}f$. By exactness of $T$,
we have an exact sequence
\begin{equation}\label{eq:coker}
TM_y\xrightarrow{Tf}TM_x\longrightarrow T(M_x/I)\longrightarrow0,
\qquad
\operatorname{im}(Tf)=TI.
\end{equation}
Since $M_y$ has simple head $L_y$ and $I\ne0$, we have
$I\twoheadrightarrow L_y$. Applying $T$ gives $TI\twoheadrightarrow TL_y\ne0$, where the nonvanishing follows from $ys=x<y$ and
\eqref{eq:am-wall-translation}. Thus $0\ne TI\subseteq TM_x$, so $Tf\ne0$ and $TM_x\ne0$.
Since $TM_x$ is a nonzero quotient of $\Delta(x\cdot\mu)$,
its highest weight space is one-dimensional and generates
$TM_x$, so $\operatorname{End}_{U(\g)}(TM_x)=\mathbb C$.
Together with (\ref{eq:transIsom}), this implies that $Tf$ is an
isomorphism. Then (\ref{eq:coker}) gives $T(M_x/I)=0$.

By Lemma~\ref{lem:am-weight-reduction} and multiplicity-freeness,
it suffices to show that every common composition factor $L_t$
of $M_x$ and $M_y$ occurs in $I$.
For such $L_t$, Lemma~\ref{lem:am-common-descent} and
\eqref{eq:am-wall-translation} give $TL_t\ne0$.
Since $T(M_x/I)=0$, exactness gives $[M_x/I:L_t]=0$,
so $L_t$ occurs in $I$.
\end{proof}

\begin{remark}
The translation step in the proof applies to arbitrary parabolics:
a nonzero map $f:M_{xs_\alpha}\to M_x$ with $xs_\alpha>x$
becomes an isomorphism after translation onto the $\alpha$-wall.
For Hermitian pairs, Lemma~\ref{lem:am-common-descent} ensures
that every common composition factor survives this translation, implying the maximal rank statement. For general parabolics, common composition factors
may be annihilated by translation, so this argument 
need not determine the image.
\end{remark}

\section{Local cohomology with support in Schubert varieties}\label{sec:local-cohomology}

We use Theorem~\ref{thm:adjacentmaps} to determine the weight
filtration on local cohomology supported in Schubert varieties.
We then deduce a formula for their Hodge rational homology level.

\subsection{The local cohomology formula}
Recall the notation $\mathscr E_x$, $\Omega_x^+$, and
$t_x(\Omega)$ from Section~\ref{sec:Verma}.
For $x,y\in\W^{\m}$ and $\Omega\in\mathscr E_y$, define
\begin{equation}\label{eq:lc-directions}
\mathscr I(x,y,\Omega)
=\left\{\alpha\in\Pi\ \middle|\
\begin{gathered}
ys_\alpha\in\W^{\m},\qquad x\leq ys_\alpha<y,\\
\Omega\cup\{\alpha\}\in\mathscr E_{ys_\alpha}
\end{gathered}\right\}.
\end{equation}
We call a pair $(y,\Omega)$, with $\Omega\in\mathscr E_y$,
\defi{terminal} if $\Omega_y^+\cap\Pi=\varnothing$.
For $x\in\W^{\m}$, $p\in\mathbb Z$, and $q\geq0$, set
\begin{equation}\label{eq:lc-terminal-pairs}
\mathscr Y_p(x;q)
=\left\{(y,\Omega)\ \middle|\
\begin{gathered}
y\in\W^{\m},\qquad x\leq y,\qquad \ell(y)=q,\\
\Omega\in\mathscr E_y,\qquad
\Omega_y^+\cap\Pi=\varnothing,\qquad
|\Omega_y^+|=p-d_X-q
\end{gathered}\right\}.
\end{equation}
These sets are the root-theoretic counterparts of
$\mc I(\a,\b,\mathbb D)$ and $\mc Y_p(\a;q)$
in \cite{schubertLC}*{Section 4.3}.

\begin{theorem}\label{thm:local-cohomology}
For every $x\in\W^{\m}$, $q\geq0$, and $p\in\mathbb Z$,
there is an isomorphism of pure Hodge modules
\begin{equation}\label{eq:lc-main-formula}
\operatorname{Gr}^W_p
\mathscr H^q_{Z_x}(\mathscr O_X^H)
\cong
\bigoplus_{\substack{
(y,\Omega)\in\mathscr Y_p(x;q)\\
\mathscr I(x,y,\Omega)=\varnothing
}}
\operatorname{IC}_{Z_{t_y(\Omega)}}^H
\left(\frac{d_X-\ell(t_y(\Omega))-p}{2}\right).
\end{equation}
\end{theorem}

\noindent The argument proceeds as follows. Using notation as in Section \ref{subsec:am-geometric-conventions}, we construct a decomposition:
\begin{equation}\label{eq:directDecomp}
\operatorname{Gr}^W_p\mathscr{GC}_x^{\bullet}    =\bigoplus_{q\geq \ell(x)}\bigoplus_{(y,\Omega)\in \mathscr{Y}_p(x;q)} \mathscr{K}_{x,y,\Omega}^{\bullet}.
\end{equation}
We prove Theorem \ref{thm:local-cohomology} by showing that $\mathscr{K}_{x,y,\Omega}^{\bullet}$ is contractible unless $\mathscr I(x,y,\Omega)=\varnothing$, and 
\begin{equation}
\mathscr{K}_{x,y,\Omega}^{\bullet}\cong \operatorname{IC}_{Z_{t_y(\Omega)}}^H
\left(\frac{d_X-\ell(t_y(\Omega))-p}{2}\right)[-q],\quad \textnormal{for $\mathscr I(x,y,\Omega)=\varnothing$.}    
\end{equation}
This argument generalizes the one in \cite{schubertLC}*{Section 4.3}.

The local cohomological defect of a subvariety $Z\subseteq X$ of pure codimension $c$ is defined via
\begin{equation}
\operatorname{lcdef}(Z)=\max\{q-c\mid \mathscr{H}^q_Z(\mathscr{O}_X)\neq 0\}.
\end{equation}
The following is an immediate application of Theorem \ref{thm:local-cohomology}.

\begin{cor}
For $Z_x\subseteq X$, the local cohomological defect is 
\[
\operatorname{lcdef}(Z_x)
=\max\left\{
q-\ell(x)\ \middle|\
\begin{gathered}
(y,\Omega)\in\mathscr Y_p(x;q)
\textnormal{ for some }p,\\
\mathscr I(x,y,\Omega)=\varnothing
\end{gathered}
\right\}.
\]
\end{cor}
\noindent Below, we calculate $\operatorname{lcdef}(Z_x)$ for the pairs $(\mathsf{B}_n,\mathsf{B}_{n-1})$, $(\mathsf{D}_n,\mathsf{D}_{n-1})$, $(\mathsf{E}_6,\mathsf{D}_5)$, $(\mathsf{E}_7,\mathsf{E}_6)$.

\subsection{Proof of the local cohomology formula} Before we prove the main result, we record a couple lemmas that will be used to establish the direct sum decomposition (\ref{eq:directDecomp}).

\begin{lemma}\label{lem:lc-deletion}
Let $x\in\W^{\m}$, $\Psi\in\mathscr E_x$, and
$\alpha\in\Psi_x^+\cap\Pi$. Then $y=xs_\alpha$ belongs to
$\W^{\m}$ and $\ell(y)=\ell(x)+1$. Define
\begin{equation}\label{eq:lc-root-deletion}
D_\alpha\Psi
=\begin{cases}
\Psi\setminus\{\alpha\},&\text{if $\alpha$ is maximal in $\Psi$},\\
\Psi,&\text{otherwise},
\end{cases}
\end{equation}
where maximality is taken in the root order. Then
\begin{equation}\label{eq:lc-deletion-identities}
D_\alpha\Psi\in\mathscr E_y,\qquad
(D_\alpha\Psi)_y^+=\Psi_x^+\setminus\{\alpha\},\qquad
t_y(D_\alpha\Psi)=t_x(\Psi).
\end{equation}
In particular, $D_\alpha\Psi$ is the unique member of
$\mathscr E_y$ with this simple label $t_x(\Psi)$.
\end{lemma}

\begin{proof}
Since $x\alpha\in\Phi(\mathfrak{u})$,
Lemma~\ref{lem:am-root-conventions}\textup{(iv)} gives
$y=xs_\alpha\in\W^{\m}$ and $\ell(y)=\ell(x)+1$.
Put $t=t_x(\Psi)$ and $J=\Psi_x^+$.
By Lemma~\ref{lem:am-column} and the injectivity in
\eqref{eq:admissible-factors}, we have
$x=x_J$ and $\Psi=\Omega_J$.
The proof of Proposition~\ref{prop:am-common} gives
$y=x_{J\setminus\{\alpha\}}$. Since $\alpha$ is simple, the only positive root
$\eta\leq\alpha$ is $\alpha$ itself. Thus
\eqref{eq:am-column} gives $\Omega_J=\Omega_{J\setminus\{\alpha\}}\cup\{\alpha\}$. If $\alpha$ is maximal in $\Omega_J$, then
$\alpha\notin\Omega_{J\setminus\{\alpha\}}$, since otherwise
$\alpha<\gamma$ for some
$\gamma\in J\setminus\{\alpha\}\subseteq\Omega_J$.
Conversely, if $\alpha$ is not maximal in $\Omega_J$,
choose $\eta\in\Omega_J$ with $\alpha<\eta$.
By \eqref{eq:am-column}, there is $\gamma\in J$
with $\eta\leq\gamma$. Then $\gamma\neq\alpha$ and
$\alpha<\gamma$, so
$\alpha\in\Omega_{J\setminus\{\alpha\}}$.
Since $\Psi=\Omega_J$, these two cases give
\[
\Omega_{J\setminus\{\alpha\}}=D_\alpha\Psi.
\]
Applying Lemma~\ref{lem:am-column} to
$J\setminus\{\alpha\}$, with
$x_{J\setminus\{\alpha\}}=y$, gives
\eqref{eq:lc-deletion-identities}.
\end{proof}

For $x\in \W^{\m}$, form a directed graph $\mathscr{Q}_x$ with vertices $(z,\Psi)$, where
$z\geq x$ and $\Psi\in\mathscr E_z$, and edges
\begin{equation}
(z,\Psi)\longrightarrow(zs_\alpha,D_\alpha\Psi)
\qquad(\alpha\in\Psi_z^+\cap\Pi).
\end{equation}
This orientation agrees with the differential in $\mathscr{GC}_x^{\bullet}$.
By Lemma \ref{lem:lc-deletion}, the source and target of each edge in $\mathscr{Q}_x$ share the simple label $t_z(\Psi)$ and the weight
$d_X+\ell(z)+|\Psi_z^+|$.

\begin{lemma}\label{lem:lc-boolean}
Every connected component of $\mathscr Q_x$ has a unique terminal
vertex $(y,\Omega)$. Its vertices are exactly $(y_J,\Psi_J)$, where
\begin{equation}\label{eq:lc-cube-vertices}
y_J=y\prod_{\alpha\in J}s_\alpha,\qquad
\Psi_J=\Omega\cup J,\qquad
J\subseteq\mathscr I(x,y,\Omega),
\end{equation}
and its edges are
\begin{equation}\label{eq:lc-cube-edges}
(y_J,\Psi_J)\longrightarrow
(y_{J\setminus\{\alpha\}},\Psi_{J\setminus\{\alpha\}})
\qquad(\alpha\in J).
\end{equation}
\end{lemma}

\begin{proof}
Let $(y,\Omega)$ be a terminal pair belonging to $\mathscr Q_x$.
We determine its connected component. Put
\[
t=t_y(\Omega),\qquad B=\Omega_y^+,\qquad
S=\mathscr R_t\cap\Pi.
\]
By Lemma~\ref{lem:am-column} and \eqref{eq:admissible-factors},
we have $(y,\Omega)=(x_B,\Omega_B)$.
Since $(y,\Omega)$ is terminal, $B\cap\Pi=\varnothing$.

The proof of Lemma~\ref{lem:lc-deletion} gives, for
$K\subseteq\mathscr R_t$ and $\alpha\in K\cap\Pi$,
\begin{equation}\label{eq:lc-column-deletion}
x_Ks_\alpha=x_{K\setminus\{\alpha\}}>x_K,
\qquad
D_\alpha\Omega_K=\Omega_{K\setminus\{\alpha\}}.
\end{equation}
The roots in $\mathscr R_t$ are pairwise orthogonal, so their
reflections commute. Starting with $x_B=y$, induction on $|J|$
using \eqref{eq:lc-column-deletion} gives
$x_{B\cup J}=y\prod_{\alpha\in J}s_\alpha$.
Since the only positive root below a simple root $\alpha$
is $\alpha$ itself, \eqref{eq:am-column} also gives
$\Omega_{B\cup J}=\Omega_B\cup J=\Omega\cup J$.
Consequently,
\[
(x_{B\cup J},\Omega_{B\cup J})
=
\left(y\prod_{\alpha\in J}s_\alpha,\ \Omega\cup J\right)
\qquad(J\subseteq S).
\]
We next show that these pairs belong to $\mathscr Q_x$
exactly when $J\subseteq\mathscr I(x,y,\Omega)$.

We claim that
\begin{equation}\label{eq:lc-directions-column}
\mathscr I(x,y,\Omega)
=\{\alpha\in S:x\leq ys_\alpha\}.
\end{equation}
For $\alpha\in S$, taking $J=\{\alpha\}$ above gives
$ys_\alpha\in\W^{\m}$, $ys_\alpha<y$, and
$\Omega\cup\{\alpha\}\in\mathscr E_{ys_\alpha}$.
If also $x\leq ys_\alpha$, then
$\alpha\in\mathscr I(x,y,\Omega)$ by \eqref{eq:lc-directions}. Conversely, let $\alpha\in\mathscr I(x,y,\Omega)$ and put
$z=ys_\alpha$ and $\Psi=\Omega\cup\{\alpha\}$.
By definition, $x\leq z<y$ and $\Psi\in\mathscr E_z$.
Lemma~\ref{lem:am-root-conventions}\textup{(iv)} gives
$z\alpha=-y\alpha\in\Phi(\mathfrak u)$.
Since $\Psi$ is pairwise orthogonal, $s_\alpha$ fixes
$\Omega\setminus\{\alpha\}$. Together with $\Omega_y^+=B$,
this gives $\Psi_z^+=B\sqcup\{\alpha\}$. Thus,
\eqref{eq:admissible-label} gives
\[
t_z(\Psi)
=\overline{ys_\alpha s_\alpha\prod_{\gamma\in B}s_\gamma}
=\overline{y\prod_{\gamma\in B}s_\gamma}=t.
\]
Lemma~\ref{lem:am-column} gives
$B\cup\{\alpha\}\subseteq\mathscr R_t$, so $\alpha\in S$.
This proves \eqref{eq:lc-directions-column}.

We next determine when $x\leq y_J$ for $J\subseteq S$.
For $\alpha\in J$, \eqref{eq:lc-column-deletion} gives
$y_Js_\alpha=y_{J\setminus\{\alpha\}}>y_J$.
The reflections indexed by $J\setminus\{\alpha\}$ fix $\alpha$,
so $y_J\alpha=-y\alpha$. Thus \eqref{eq:root-cover-ideal} gives
\[
\Phi_{y_J}
=\Phi_{y_{J\setminus\{\alpha\}}}\setminus\{-y\alpha\}.
\]
Induction on $|J|$ gives
\begin{equation}\label{eq:lc-face-ideal}
\Phi_{y_J}=\Phi_y\setminus\{-y\alpha\mid \alpha\in J\}.
\end{equation}
By Lemma~\ref{lem:root-bruhat-ideals},
$x\leq y$ gives $\Phi_x\subseteq\Phi_y$.
Thus \eqref{eq:lc-face-ideal} gives
\[
x\leq y_J
\quad\Longleftrightarrow\quad
\Phi_x\subseteq\Phi_{y_J}
\quad\Longleftrightarrow\quad
-y\alpha\notin\Phi_x\text{ for every }\alpha\in J.
\]
For $\alpha\in S$, applying \eqref{eq:lc-face-ideal}
with $J=\{\alpha\}$ gives
$\Phi_{ys_\alpha}=\Phi_y\setminus\{-y\alpha\}$.
Since $\Phi_x\subseteq\Phi_y$, the same lemma shows that
$-y\alpha\notin\Phi_x$ is equivalent to $x\leq ys_\alpha$.
Consequently, \eqref{eq:lc-directions-column} gives
\[
x\leq y_J
\quad\Longleftrightarrow\quad
x\leq ys_\alpha\text{ for every }\alpha\in J
\quad\Longleftrightarrow\quad
J\subseteq\mathscr I(x,y,\Omega).
\]

Every edge preserves the simple label $t$.
By Lemma~\ref{lem:am-column} and \eqref{eq:lc-column-deletion},
a vertex adjacent to $(x_{B\cup J},\Omega_{B\cup J})$
can only be obtained by adjoining or deleting a root of $S$
from $J$. Hence no edge of $\mathscr Q_x$ joins the pairs
in \eqref{eq:lc-cube-vertices} to any other vertex.
Their edges are precisely \eqref{eq:lc-cube-edges},
which connect every $J$ to $\varnothing$.
Thus they form the connected component of $(y,\Omega)$,
and $(y,\Omega)$ is its unique terminal vertex.

Finally, every edge of $\mathscr Q_x$ increases
the length of the Weyl group index by one, so every vertex
has a path to a terminal vertex.
Therefore every vertex belongs to one of the components
above.
\end{proof}

\begin{proof}[Proof of Theorem~\ref{thm:local-cohomology}]
Fix $x\in\W^{\m}$ and $p\in\mathbb Z$. For $t\in\W^{\m}$ with $d_X-\ell(t)\equiv p\pmod2$,
we use the shorthand
\begin{equation}\label{eq:lc-tate-twist}
\mathscr L_{t,p}
=\operatorname{IC}_{Z_t}^H
\left(\frac{d_X-\ell(t)-p}{2}\right).
\end{equation}
By \eqref{eq:lc-cousin-terms}, \eqref{eq:am-weight-layers}, and
localization, the simple summands of
$\operatorname{Gr}^W_p\mathscr{GC}_x^\bullet$ are indexed by the
vertices $(z,\Psi)$ of $\mathscr Q_x$ satisfying
\[
d_X+\ell(z)+|\Psi_z^+|=p.
\]
Lemmas~\ref{lem:lc-deletion} and \ref{lem:lc-boolean} partition
these vertices into connected components indexed by
$(y,\Omega)\in\mathscr Y_p(x;q)$ for $q\geq\ell(x)$.
Under localization, \eqref{eq:adjacent-weight-image},
\eqref{eq:am-common-positive-root}, and
\eqref{eq:lc-deletion-identities} identify the nonzero differential
components with the edges of $\mathscr Q_x$. Each is an
isomorphism between simple pure Hodge modules. Therefore, the
sum of the simple summands indexed by the connected component
of $(y,\Omega)$ is a subcomplex, which we denote by
$\mathscr K_{x,y,\Omega}^\bullet$.
This proves the decomposition \eqref{eq:directDecomp}.

For $(y,\Omega)\in\mathscr Y_p(x;q)$, put
$I=\mathscr I(x,y,\Omega)$ and $t=t_y(\Omega)$.
By \eqref{eq:lc-cube-vertices} and Lemma~\ref{lem:lc-deletion},
the summand indexed by $J\subseteq I$ is $\mathscr L_{t,p}$
and occurs in degree $q-|J|$. Thus
\begin{equation}\label{eq:lc-cube-terms}
\mathscr K_{x,y,\Omega}^j
\cong\bigoplus_{\substack{J\subseteq I\\|J|=q-j}}
\mathscr L_{t,p}.
\end{equation}

If $I=\varnothing$, \eqref{eq:lc-cube-terms} gives
$\mathscr K_{x,y,\Omega}^{\bullet}\cong\mathscr L_{t,p}[-q]$. If $I\ne\varnothing$, Theorem~\ref{thm:adjacentmaps} and
Lemma~\ref{lem:lc-boolean} identify
$\mathscr K_{x,y,\Omega}^{\bullet}$ with $\mathscr L_{t,p}$ tensored with the Koszul complex on
$(1)_{\alpha\in I}$ over $\mathbb Q$, shifted by $[-q]$.
Hence it is contractible. Taking cohomology in
\eqref{eq:directDecomp}, strictness of $W_\bullet$ and
\eqref{eq:lc-cousin-terms} give \eqref{eq:lc-main-formula}.
\end{proof}

\subsection{Hodge rational homology level}

Let $Z\subseteq X$ be a closed subvariety of pure codimension $c$,
and write $F_\bullet$ for the Hodge filtration on $\mathscr H_Z^q(\mathscr O_X^H)$. The lowest weight piece of
$\mathscr H_Z^c(\mathscr O_X^H)$ is
$\operatorname{IC}_Z^H(-c)$. The Hodge rational homology level
of \cites{DORhrh, ParkPopa}, written $\operatorname{HRH}(Z)$, is an invariant
that compares the Hodge filtration $F_{\bullet}$ on local cohomology to that
on this submodule. Formally,
\begin{equation}\label{eq:hrh-criterion}
 \operatorname{HRH}(Z)\geq k\quad\Longleftrightarrow\quad
 F_k\mathscr H_Z^c(\mathscr O_X^H)
   =F_k\operatorname{IC}_Z^H(-c)\textnormal{ and }
 F_k\mathscr H_Z^q(\mathscr O_X^H)=0\textnormal{ for $q>c$, }
\end{equation}
where the equality is under the natural inclusion.
We set $\operatorname{HRH}(Z)=-1$ if the condition fails for
$k=0$, and $\operatorname{HRH}(Z)=+\infty$ if it holds for
every $k$. 

Put $c=\ell(x)$. By Theorem \ref{thm:local-cohomology}, $Z_x$ is a
rational homology manifold if and only if, for all $q\geq 0$ and $p\geq c+d_X+1$, we have $\mathscr{I}(x,y,\Omega)\neq \varnothing$ for all $(y,\Omega)\in \mathscr{Y}_p(x;q)$. For a survey of the rational homology manifold property on compact Hermitian symmetric spaces, see \cite[Sections 4.3, 5.3]{EHP}. We refine this by describing the Hodge rational homology level for Schubert varieties.

\begin{cor}\label{cor:hrh}
For $x\in\W^{\m}$, put $c=\ell(x)$. Then
\begin{equation}\label{eq:hrh-formula}
 \operatorname{HRH}(Z_x)
 =\min\left\{
       \frac{\ell(t_y(\Omega))-\ell(y)-|\Omega_y^+|}{2}-1 \;\middle|\; \substack{q\geq c,\;\;p\geq c+d_X+1\\ (y,\Omega)\in  \mathscr{Y}_p(x;q)\\
                         \mathscr I(x,y,\Omega)=\varnothing} \right\},
\end{equation}
with the convention that the minimum of the empty set is $+\infty$.
\end{cor}
\begin{proof}
For $t\in \W^{\m}$, the first nonzero Hodge piece of
$\operatorname{IC}_{Z_t}^H(k)$ has index $\ell(t)+k$ (see, for instance, \cite{PRdet}*{Lemma 2.1}). Thus the constituent
indexed by $(y,\Omega)$ in \eqref{eq:lc-main-formula} starts in level
\begin{equation}\label{eq:hrh-first-level}
\frac{\ell(t_y(\Omega))-\ell(y)-|\Omega_y^+|}{2}.
\end{equation}
By strictness of $F_\bullet$ and \eqref{eq:hrh-criterion}, this verifies (\ref{eq:hrh-formula}), noting that $\operatorname{IC}^H_{Z_x}(-c)$ is the only composition factor of local cohomology with weight $\leq c+d_X$.
\end{proof}

\section{The pairs $(\mathsf{C}_n,\mathsf{A}_{n-1})$ and $(\mathsf{D}_n,\mathsf{A}_{n-1})$}\label{sec:shifted-pairs}

In Section~\ref{sec:shifted-proof}, we prove Theorem~\ref{thm:sh-local-cohomology}.
In Section~\ref{sec:matrix-applications}, we restrict to the
opposite big cell to calculate local cohomology supported in
symmetric determinantal and Pfaffian varieties, recovering
the composition factors of \cite{raicu2016local} and
determining the weight filtration.

\subsection{Proof of Theorem~\ref{thm:sh-local-cohomology}}
\label{sec:shifted-proof}

We use the notation of Section~\ref{sec:shifted-shapes}.
Put $\delta_m=(m,m-1,\ldots,1)$, and write $\mc S_m$ for
the strict partitions contained in $\delta_m$, including
the empty partition. Their shifted diagrams are precisely
the order ideals of $\delta_m$ in the coordinatewise order.
We say that a box $z\in\b$ is a \defi{removable corner} of $\b$ if
$\b\setminus\{z\}\in\mc S_m$, and a box
$z\in\delta_m\setminus\b$ is an \defi{addable corner} of $\b$
if $\b\cup\{z\}\in\mc S_m$. 

Every completed strip $\theta_s$ that occurs in a shifted Dyck shape $\eta$ has odd
size. Indeed, with notation as in \eqref{eq:sh-completion},
the subpath $\{x_1,\ldots,x_r\}\subseteq \theta$ has endpoints of equal level,
so $r$ is odd. The condition in \eqref{eq:sh-dyck-recursion}
says that $t-r$ is even. Thus $|\theta|=t$ and
$|\theta_s|=2t-r$ are odd, and hence
\begin{equation}\label{eq:sh-depth-parity}
|\eta|\equiv\operatorname{dp}(\eta)\pmod2.
\end{equation}
In particular, if $\eta$ has no singleton strips, then $|\theta|\geq3$ for all almost Dyck paths $\theta$ in $\eta$.

\begin{lemma}\label{lem:sh-corner-removal}
Let $\underline c\subseteq\b\in\mc S_m$ and suppose
$\eta=\b/\underline c\in\operatorname{sDyck}(\b)$.
Let $S(\eta)$ be the set of boxes removed from $\eta$ as singleton strips
in \eqref{eq:sh-dyck-recursion}. Then
\begin{equation}\label{eq:sh-corner-removal}
S(\eta)=\left\{z\in\eta\ \middle|\
\begin{gathered}
\b\setminus\{z\}\in\mc S_m,\\
\eta\setminus\{z\}\in
\operatorname{sDyck}(\b\setminus\{z\})
\end{gathered}\right\}.
\end{equation}
\end{lemma}

\begin{proof}
We first show that the left side of (\ref{eq:sh-corner-removal}) is contained in the right side. Let $z\in S(\eta)$, and let $C$ be the connected component in the recursive definition of shifted Dyck shape for which $\{z\}=\theta(C)$. By definition of border strip, we have $C=\{z\}$. It follows that no box in $\eta$ to the East nor South of $z$ belongs to $\eta$. Since $\eta=\b\setminus \underline{c}$ and $\underline{c}$ is a strict partition, we have that $\b\setminus \{z\}\in \mathcal{S}_m$. Furthermore, since $\eta\setminus \{z\}=\eta\setminus C$, by the recursive definition of shifted Dyck shape, $\eta\setminus\{z\}\in
\operatorname{sDyck}(\b\setminus\{z\})$.

Conversely, assume that $z$ is a box belonging to the right side of (\ref{eq:sh-corner-removal}). Since $\eta\in \operatorname{sDyck}(\b)$, there is a connected component $C$ in the recursive definition of shifted Dyck shape for which $z\in \theta_s$, where $\theta=\theta(C)$ is the border strip of $C$. We must show that $\theta_s=\{z\}$. Since $\b\setminus \{z\}\in \mathcal{S}_m$, it follows that no box South or East of $z$ belongs to $\b$. Thus, no box South or East of $z$ belongs to $\eta$. Since $\eta\setminus \{z\}$ is shifted Dyck, it follows that $C\setminus \{z\}$ is shifted Dyck. If $z\in E(\theta)$, deleting $z$ leaves the
border $\theta$ unchanged, but its completion requires
the missing box $z$, contrary to
\eqref{eq:sh-dyck-recursion}. Thus $z\in \theta$. Write $\theta=\{x_1,\ldots,x_t\}$, ordered from Northeast to
Southwest, and assume $z=x_k$. If $k=1$ and $t>1$, then $z$ has only one neighbor in $C$, and
deleting it leaves a connected shape with border
$\{x_2,\ldots,x_t\}$. This border has even size, whereas
the border of a connected shifted Dyck shape has odd size (see discussion at beginning of Section \ref{sec:shifted-proof}), contradicting that $C\setminus \{z\}$ is shifted Dyck. If $k>1$, then $x_{k-1}$ lies directly north of $z$, since
$z$ is a removable corner. The component of
$C\setminus\{z\}$ containing $x_{k-1}$ has border
$\{x_1,\ldots,x_{k-1}\}$. By \eqref{eq:sh-almost-dyck},
\[
\operatorname{lv}(x_{k-1})
=\operatorname{lv}(z)-1
<\operatorname{lv}(x_1).
\]
Since $x_{k-1}$ has level strictly less than
$\operatorname{lv}(x_1)$, the last box of the border
$\{x_1,\ldots,x_{k-1}\}$ at level
$\operatorname{lv}(x_1)$ occurs before $x_{k-1}$.
Thus \eqref{eq:sh-completion} requires the shifted
completion of this border to contain
\[
x_{k-1}+(1,1)=z+(0,1)\notin\b,
\]
contrary to \eqref{eq:sh-dyck-recursion}. Therefore, $\theta_s=\{z\}$, so $z\in S(\eta)$.
\end{proof}

\begin{remark}\label{rem:sh-brenti-conventions}
We explain the translation between our coordinates and \cite{brenti}. For $(i,j)\in\delta_m$, set
\[
\iota_m(i,j)=(m+1-j,m+1-i).
\]
The map $\iota_m$ interchanges South and West steps, taking our diagrams to Brenti's rotated diagrams. As Brenti's level function is
$\operatorname{lv}_{\mathrm{Br}}(i,j)=i+j-1$, for a connected skew shifted shape $\eta$, we have
\[
\theta_{\mathrm{Br}}(\iota_m(\eta))=\iota_m(\theta(\eta)),
\qquad
\operatorname{lv}_{\mathrm{Br}}(\iota_m(u))
=2m+1-\operatorname{lv}(u).
\]
The inequality in \eqref{eq:sh-almost-dyck} therefore becomes
his inequality
\[
\operatorname{lv}_{\mathrm{Br}}(\iota_m(x_u))
\geq\operatorname{lv}_{\mathrm{Br}}(\iota_m(x_1)),
\]
and the last return index $r$ is unchanged. Consequently, \eqref{eq:sh-dyck-recursion} is the recursion
of \cite{brenti}*{Section 3} after applying $\iota_m$,
and the depths agree.
\end{remark}

For $\b\in\mc S_m$, let $w_{\b}\in\W^{\m}$ be the element
whose partition in \cite{brenti}*{Propositions 2.5 and 2.9} is
\[
\b^\vee:=\delta_m\setminus\iota_m(\b),
\]
with $\iota_m$ as in Remark~\ref{rem:sh-brenti-conventions}. Then
\begin{equation}\label{eq:sh-schubert-indices}
Z_{\b}=Z_{w_{\b}},\qquad
\ell(w_{\b})=d_X-|\b|,\qquad
w_{\a}\leq w_{\b}
\quad\Longleftrightarrow\quad\b\subseteq\a.
\end{equation}

\begin{proof}[Proof of Theorem~\ref{thm:sh-local-cohomology}]
Fix $\a\in\mc S_m$ and put $x=w_{\a}$.
For $\underline c\subseteq\b\subseteq\a$, put
$y=w_{\b}$, $t=w_{\underline c}$, and
$\eta=\b/\underline c$.
Since $\iota_m(\eta)=\underline c^\vee/\b^\vee$,
Remark~\ref{rem:sh-brenti-conventions} allows us to apply
\cite{brenti}*{Theorem 4.1}.
Together with \cite{irvingBook}*{Corollary 7.1.3},
BGG duality, and the normalization in
Proposition~\ref{prop:am-weight-normalization}, this gives
\begin{equation}\label{eq:sh-standard-factors}
\sum_{r\geq0}
\big[\operatorname{Gr}^W_{d_X+\ell(y)+r}N_y:L_t\big]v^r
=\begin{cases}
v^{\operatorname{dp}(\eta)},
   &\eta\in\operatorname{sDyck}(\b),\\
0,&\textnormal{otherwise}.
\end{cases}
\end{equation}
Every simple constituent of $N_y$ is supported on $Z_{\b}$
after localization, so \eqref{eq:sh-schubert-indices} gives
all its possible labels as $w_{\underline c}$ with
$\underline c\subseteq\b$.
Comparing \eqref{eq:sh-standard-factors} with
\eqref{eq:am-weight-layers} and using
\eqref{eq:admissible-factors} gives
\begin{equation}\label{eq:sh-root-depth}
\begin{gathered}
\mathscr E_y\xrightarrow{\sim}\operatorname{sDyck}(\b),\qquad
\Omega\longmapsto\b/\underline c,
\quad t_y(\Omega)=w_{\underline c},\\
|\Omega_y^+|=\operatorname{dp}(\b/\underline c).
\end{gathered}
\end{equation}

Let $\Omega\in\mathscr E_y$ correspond to $\eta$ under
\eqref{eq:sh-root-depth}.
For a removable corner $u\in\eta$ of $\b$, write
$ys_\alpha=w_{\b\setminus\{u\}}>y$, using
\eqref{eq:sh-schubert-indices} and
Lemma~\ref{lem:root-bruhat-ideals}\textup{(iii)}.
Equations \eqref{eq:am-common},
\eqref{eq:am-common-positive-root}, and
\eqref{eq:sh-standard-factors}, together with
Lemma~\ref{lem:sh-corner-removal}, give
\[
\alpha\in\Omega_y^+\cap\Pi
\quad\Longleftrightarrow\quad
\eta\setminus\{u\}\in
\operatorname{sDyck}(\b\setminus\{u\})
\quad\Longleftrightarrow\quad u\in S(\eta).
\]
Conversely, every $\alpha\in\Omega_y^+\cap\Pi$ gives a
removable corner $u\in\b$ with
$ys_\alpha=w_{\b\setminus\{u\}}$, by
Lemma~\ref{lem:lc-deletion} and \eqref{eq:sh-schubert-indices}.
Equations \eqref{eq:lc-deletion-identities} and
\eqref{eq:sh-root-depth} imply
$\underline c\subseteq\b\setminus\{u\}$, so $u\in\eta$.
Therefore
\begin{equation}\label{eq:sh-terminal-condition}
\Omega_y^+\cap\Pi=\varnothing
\quad\Longleftrightarrow\quad
\eta\textnormal{ has no singleton strips}.
\end{equation}

Suppose $\Omega_y^+\cap\Pi=\varnothing$.
For an addable corner $u\in\a/\b$ of $\b$,
\eqref{eq:sh-schubert-indices} and
Lemma~\ref{lem:root-bruhat-ideals}\textup{(iii)} give
\[
z=ys_\alpha=w_{\b\cup\{u\}},\qquad x\leq z<y.
\]
By \eqref{eq:sh-root-depth},
$(\b\cup\{u\})/\underline c$ is shifted Dyck if and only if
there is $\Psi\in\mathscr E_z$ with $t_z(\Psi)=t_y(\Omega)$.
For such $\Psi$, \eqref{eq:am-common} and
\eqref{eq:am-common-positive-root} give $\alpha\in\Psi_z^+$.
Thus Lemma~\ref{lem:lc-deletion} gives $D_\alpha\Psi=\Omega$.
Equation \eqref{eq:lc-root-deletion} then gives
$\Psi=\Omega\cup\{\alpha\}$, so
$\alpha\in\mathscr I(x,y,\Omega)$ by \eqref{eq:lc-directions}.
Conversely, if $\alpha\in\mathscr I(x,y,\Omega)$,
Lemma~\ref{lem:lc-boolean} places
$(z,\Omega\cup\{\alpha\})$ in the connected component of
$(y,\Omega)$, and \eqref{eq:lc-deletion-identities} gives
$t_z(\Omega\cup\{\alpha\})=t_y(\Omega)$.
Hence \eqref{eq:sh-root-depth} gives
\[
\alpha\in\mathscr I(x,y,\Omega)
\quad\Longleftrightarrow\quad
(\b\cup\{u\})/\underline c\in
\operatorname{sDyck}(\b\cup\{u\}).
\]
Every element of $\mathscr I(x,y,\Omega)$ arises from such
a corner by \eqref{eq:lc-directions} and
\eqref{eq:sh-schubert-indices}.
Therefore \eqref{eq:sh-bullet-condition} gives
\begin{equation}\label{eq:sh-bullet-equivalence}
\mathscr I(x,y,\Omega)=\varnothing
\quad\Longleftrightarrow\quad
(\eta;\a/\b)\in\operatorname{sDyck}^{\bullet}(\a).
\end{equation}

For $\DD=(\eta;\a/\b)$, equations
\eqref{eq:sh-schubert-indices},
\eqref{eq:sh-augmented-support},
\eqref{eq:sh-augmented-statistics}, and
\eqref{eq:sh-root-depth} give
\begin{equation}\label{eq:sh-root-statistics}
\ell(y)=d_X-|\a|+b(\DD),\qquad
|\Omega_y^+|=\operatorname{dp}(\DD),\qquad
t_y(\Omega)=w_{\a^{\DD}}.
\end{equation}
By \eqref{eq:lc-terminal-pairs},
\eqref{eq:sh-terminal-condition}, and
\eqref{eq:sh-bullet-equivalence}, we obtain a bijection
\begin{equation}\label{eq:sh-terminal-bijection}
\begin{gathered}
\left\{(y,\Omega)\in\mathscr Y_p(x;q)\ \middle|\
\mathscr I(x,y,\Omega)=\varnothing\right\}
\xrightarrow{\sim}\mathscr A^{\mathrm{sh}}_p(\a;q),\\
(y,\Omega)\longmapsto\DD=(\b/\underline c;\a/\b).
\end{gathered}
\end{equation}
Indeed, \eqref{eq:sh-root-statistics} identifies
$\ell(y)=q$ and $|\Omega_y^+|=p-d_X-q$ with the conditions
in \eqref{eq:sh-lc-patterns} and \eqref{eq:sh-weight-patterns}.
For the inverse, $\DD=(\eta;\mathbb B)$ determines
$\b=\a\setminus\mathbb B$, $\underline c=\a^{\DD}$,
and the unique $\Omega$ in \eqref{eq:sh-root-depth}.
Finally, \eqref{eq:sh-schubert-indices} and
\eqref{eq:sh-root-statistics} give
\[
\frac{d_X-\ell(t_y(\Omega))-p}{2}
=\frac{|\a^{\DD}|-p}{2}.
\]
Thus \eqref{eq:lc-main-formula} and
\eqref{eq:sh-terminal-bijection} give \eqref{eq:sh-weight-formula}.
\end{proof}

\subsection{Application to rank varieties of symmetric and skew-symmetric matrices}
\label{sec:matrix-applications}

Let $\mathscr X=\operatorname{Sym}^2\mathbb C^n$ or
$\mathscr X=\bigwedge^2\mathbb C^n$, and let
$\mathscr Z_p\subseteq\mathscr X$ be the variety of matrices of
rank at most $p$ or $2p$, respectively. Write
$d=\dim\mathscr X$ and let $D_r$ be the intersection cohomology
$\mathscr D_{\mathscr X}$-module of $\mathscr Z_r$ associated
to the trivial local system. The composition factors of
$\mathscr H^q_{\mathscr Z_p}(\mathscr O_{\mathscr X})$
were calculated by Raicu--Weyman
\cite{raicu2016local}*{Main Theorem}. We use
Theorem~\ref{thm:sh-local-cohomology} to determine their weights.

\begin{theorem}\label{thm:matrix-rank-varieties}
Suppose $0\leq p<n$ in the symmetric case and
$0\leq p<\lfloor n/2\rfloor$ in the skew-symmetric case.
For every $q\geq0$, each copy of $D_r$ in
$\mathscr H^q_{\mathscr Z_p}(\mathscr O_{\mathscr X})$
underlies a pure Hodge module of weight
\begin{equation}\label{eq:matrix-weight-factor}
\begin{cases}
d+q+\dfrac{p-r}{2},&\mathscr X=\operatorname{Sym}^2\mathbb C^n,\\[4pt]
d+q+p-r,&\mathscr X=\bigwedge^2\mathbb C^n.
\end{cases}
\end{equation}
\end{theorem}

Let $X=\operatorname{LGr}(n,2n)$ or $\operatorname{OGr}(n,2n)$, respectively, and let $F\in X$
be the closed Schubert point. Choosing an isotropic complement
$F^*$ gives an open chart $U\cong\mathscr X$ whose points are
graphs $V_A$ of symmetric or skew-symmetric maps $A:F\to F^*$.
Since $V_A\cap F=\ker A$, the rank loci are the restrictions of
the Schubert varieties indexed by
\begin{equation}\label{eq:matrix-rank-diagrams}
\underline p=
\begin{cases}
(n,n-1,\ldots,n-p+1),&\mathscr X=\operatorname{Sym}^2\mathbb C^n,\\
(n-1,n-2,\ldots,n-2p),&\mathscr X=\bigwedge^2\mathbb C^n.
\end{cases}
\end{equation}
Thus $\underline p$ consists of the first $p$ or $2p$ rows of
$\delta_m$, where $m=n$ or $n-1$ as before; $\underline0$ is
the zero diagram. Open restriction gives
\begin{equation}\label{eq:matrix-open-restriction}
\operatorname{IC}_{Z_{\underline r}}^H|_U
 \cong\operatorname{IC}_{\mathscr Z_r}^H,
\qquad
\mathscr H^q_{Z_{\underline p}}(\mathscr O_X^H)|_U
 \cong\mathscr H^q_{\mathscr Z_p}(\mathscr O_{\mathscr X}^H).
\end{equation}

\begin{proof}[Proof of Theorem~\ref{thm:matrix-rank-varieties}]
Write $k=p$, $s=r$ in the symmetric case and $k=2p$, $s=2r$
in the skew-symmetric case. Let
$\DD=(\b/\underline r;\underline p/\b)
\in\mathscr A^{\mathrm{sh}}(\underline p;q)$.
Put $\eta=\b/\underline r$. Since $\underline r$ consists
of the first $s$ rows of $\delta_m$, $\eta$ consists of
the rows of $\b$ with indices greater than $s$.
If $\eta\ne\varnothing$, it is connected. Writing $L$ for
the length of its first row, \eqref{eq:sh-inner-boundary}
and \eqref{eq:sh-completion} give
\[
\theta=\theta(\eta)
 =\{(s+1,j)\mid s+1\leq j\leq s+L\},
\qquad
E(\theta)
 =\{(s+2,j)\mid s+2\leq j\leq s+L\}.
\]
By \eqref{eq:sh-dyck-recursion}, $|E(\theta)|=L-1$ is
even, so $L$ is odd. The absence of singleton strips
excludes $L=1$, hence $L\geq3$.
The containment $E(\theta)\subseteq\eta$ requires row $s+2$ to have length at least $L-1$, while strictness
of $\b$ makes its length at most $L-1$.
Thus $\theta_s$ consists of exactly the first two rows
of $\eta$, of lengths $(L,L-1)$.
The remaining shape is again shifted Dyck with no
singleton strips. Induction therefore pairs all rows
of $\eta$, with each pair having lengths $(2h+1,2h)$
for some $h\geq1$.
Moreover, $\b$ has exactly $k$ rows: otherwise, adjoining
the next diagonal box gives a shifted Dyck shape with a
final singleton strip, contrary to
\eqref{eq:sh-bullet-condition}. Thus $k-s$ is even and
\begin{equation}\label{eq:matrix-paired-rows}
\b=(m,m-1,\ldots,m-s+1,
       2h_1+1,2h_1,\ldots,2h_b+1,2h_b),
\qquad b=\frac{k-s}{2},
\end{equation}
where, if $b>0$, we have $h_1>\cdots>h_b\geq1$ and
$2h_1+1\leq m-s$. Conversely, every such diagram gives
an admissible pattern: each addable box in
$\underline p/\b$ extends an odd row to even length,
so \eqref{eq:sh-bullet-condition} holds.

Each pair contributes one completed strip and $4h_i+1$
boxes. Hence
\begin{equation}\label{eq:matrix-pattern-statistics}
\operatorname{dp}(\DD)=b,\qquad
q=d-|\b|=d-|\underline r|-b-4\sum_{i=1}^b h_i.
\end{equation}
By Theorem~\ref{thm:sh-local-cohomology} and
\eqref{eq:matrix-open-restriction}, the corresponding copy
of $D_r$ has weight $d+q+b$, as required.
\end{proof}

\begin{remark}
Counting the patterns in \eqref{eq:matrix-paired-rows} by
cohomological degree recovers the composition factor
formula \cite{raicu2016local}*{Main Theorem}, just as in
\cite{schubertLC}*{Section 4.5} for generic determinantal
varieties.
\end{remark}

\begin{example}
Take $\mathscr X=\operatorname{Sym}^2\mathbb C^9$ with $p=4$,
or $\mathscr X=\bigwedge^2\mathbb C^{10}$ with $p=2$.
In both cases, $d=45$ and $\underline p=(9,8,7,6)$.
The six patterns with remaining diagram $\underline0$
are shown below.\\

\begingroup
\newcommand{\shRankGrid}{%
  \foreach \shrow/\shlen in {1/9,2/8,3/7,4/6}{%
    \pgfmathtruncatemacro{\shlast}{\shrow+\shlen-1}%
    \foreach \shcol in {\shrow,...,\shlast}{%
      \draw[thick] ({2*(\shcol-1)},{-2*(\shrow-1)})
        rectangle ++(2,-2);
    }%
  }%
}
\newcommand{\shRankPair}[2]{%
  \draw[red,line width=4pt]
    ({2*(#1-1)+.4},{1-2*#1})--
    ({2*(#1-1)+2*#2-.4},{1-2*#1});
  \draw[orange,line width=4pt]
    ({2*#1+.4},{-1-2*#1})--
    ({2*(#1-1)+2*#2-.4},{-1-2*#1});
}
\newcommand{\shRankBullets}[2]{%
  \foreach \shrow in {#1}{%
    \foreach \shcol in {#2}{%
      \fill[green!65!black] ({2*\shcol-1},{1-2*\shrow})
        circle[radius=.45];
    }%
  }%
}
\begin{center}
\begin{minipage}[t]{.32\textwidth}
\centering
\begin{tikzpicture}[x=\unitsize,y=\unitsize]
\shRankGrid
\shRankPair{1}{9}\shRankPair{3}{7}
\end{tikzpicture}
\par\smallskip $q=15$
\end{minipage}\hfill
\begin{minipage}[t]{.32\textwidth}
\centering
\begin{tikzpicture}[x=\unitsize,y=\unitsize]
\shRankGrid
\shRankPair{1}{9}\shRankPair{3}{5}
\shRankBullets{3,4}{8,9}
\end{tikzpicture}
\par\smallskip $q=19$
\end{minipage}\hfill
\begin{minipage}[t]{.32\textwidth}
\centering
\begin{tikzpicture}[x=\unitsize,y=\unitsize]
\shRankGrid
\shRankPair{1}{9}\shRankPair{3}{3}
\shRankBullets{3,4}{6,7,8,9}
\end{tikzpicture}
\par\smallskip $q=23$
\end{minipage}
\par\medskip
\begin{minipage}[t]{.32\textwidth}
\centering
\begin{tikzpicture}[x=\unitsize,y=\unitsize]
\shRankGrid
\shRankPair{1}{7}\shRankPair{3}{5}
\shRankBullets{1,2,3,4}{8,9}
\end{tikzpicture}
\par\smallskip $q=23$
\end{minipage}\hfill
\begin{minipage}[t]{.32\textwidth}
\centering
\begin{tikzpicture}[x=\unitsize,y=\unitsize]
\shRankGrid
\shRankPair{1}{7}\shRankPair{3}{3}
\shRankBullets{1,2}{8,9}\shRankBullets{3,4}{6,7,8,9}
\end{tikzpicture}
\par\smallskip $q=27$
\end{minipage}\hfill
\begin{minipage}[t]{.32\textwidth}
\centering
\begin{tikzpicture}[x=\unitsize,y=\unitsize]
\shRankGrid
\shRankPair{1}{5}\shRankPair{3}{3}
\shRankBullets{1,2,3,4}{6,7,8,9}
\end{tikzpicture}
\par\smallskip $q=31$
\end{minipage}
\end{center}
\endgroup

\medskip

Each pattern contributes a copy of $D_0$ in the indicated
cohomological degree. All six have depth two, so their
weights are, from left to right in each row,
\[
62,\,66,\,70,\,70,\,74,\,78.
\]
Since $c=\operatorname{codim}\mathscr Z_p=15$,
the weight of each copy of $D_0$ is $d+c+2=62$ plus
the number of bullets. \hfill\mbox{$\diamond$}
\end{example}

The following corollary refines the value in
\cite{DORhrh}*{Proposition 10.10}, which gives
$\operatorname{HRH}(\mathscr Z_p)\in\{0,1\}$ outside the
rational homology manifold cases.

\begin{cor}\label{cor:matrix-hrh}
Under the hypotheses of
Theorem~\ref{thm:matrix-rank-varieties},
\[
\operatorname{HRH}(\mathscr Z_p)=
\begin{cases}
+\infty,
   &\mathscr X=\operatorname{Sym}^2\mathbb C^n,\ p\leq1,\\
+\infty,
   &\mathscr X=\bigwedge^2\mathbb C^n,\ p=0,\\
1,&\textnormal{otherwise}.
\end{cases}
\]
\end{cor}
\begin{proof}
Put $c=\operatorname{codim}_{\mathscr X}\mathscr Z_p$.
The pattern of depth zero gives
$\operatorname{IC}_{\mathscr Z_p}^H(-c)$ in degree $c$.
By the proof of Corollary~\ref{cor:hrh} and
\eqref{eq:matrix-pattern-statistics}, every other constituent
has Hodge filtration starting in level
\[
\frac{d-|\underline r|-q-b}{2}
=2\sum_{i=1}^b h_i\geq2.
\]
For $p\geq2$ in the symmetric case, take $r=p-2$;
for $p\geq1$ in the skew-symmetric case, take $r=p-1$.
In either case, the choice $b=1$, $h_1=1$ in
\eqref{eq:matrix-paired-rows} gives a constituent starting
in Hodge level two. In the remaining cases there are no
other constituents. The result follows from
\eqref{eq:hrh-criterion}.
\end{proof}

\section{The pairs $(\mathsf{B}_n,\mathsf{B}_{n-1})$ and $(\mathsf{D}_n,\mathsf{D}_{n-1})$}\label{sec:quadrics}

Write $Q^d\subseteq\mathbb P(V)$ for a smooth projective
quadric, where $\dim V=d+2$.
The pairs $(\mathsf B_n,\mathsf B_{n-1})$ and
$(\mathsf D_n,\mathsf D_{n-1})$ give $Q^{2n-1}$ and
$Q^{2n-2}$, respectively. See \cite[Section 9.2.28]{billey} for the assertions made here.
Fix a $B$-stable complete isotropic flag $F_\bullet$ of $V$. The singular Schubert varieties are
\begin{equation}\label{eq:quadric-exceptions}
Z_{x_c}=Q^d\cap\mathbb P(F_c^\perp),
\qquad
1\leq c\leq\left\lfloor\frac{d-1}{2}\right\rfloor,
\end{equation}
where $x_c$ denotes the corresponding Schubert index.
Each has codimension $c$ and singular locus
\[
\operatorname{Sing}(Z_{x_c})
=S_c:=\mathbb P(F_c)\cong\mathbb P^{c-1}.
\]
The remaining Schubert varieties are smooth. Indeed, they are $Q^d$
itself and linear spaces $\mathbb P^k$ for
$0\leq k\leq\lfloor d/2\rfloor$, with two of dimension
$n-1$ when $d=2n-2$.

\begin{prop}\label{prop:quadric-cohomology}
Every Schubert variety in $Q^{2n-1}$ is a rational homology
manifold. A Schubert variety in $Q^{2n-2}$ is a rational homology
manifold if and only if it is smooth.

For every Schubert variety $Z\subseteq Q^d$ of codimension $c$,
we have $\mathscr H_Z^q(\mathscr O_{Q^d}^H)=0$ for $q\ne c$.
For $Z=Z_{x_c}\subseteq Q^{2n-2}$ with $1\leq c\leq n-2$,
the nonzero weight pieces are
\begin{equation}\label{eq:quadric-weights}
 \operatorname{Gr}_{2n-2+c}^W
   \mathscr H_{Z_{x_c}}^c(\mathscr O_{Q^{2n-2}}^H)
   \cong\operatorname{IC}_{Z_{x_c}}^H(-c),\quad
 \operatorname{Gr}_{2n-1+c}^W
   \mathscr H_{Z_{x_c}}^c(\mathscr O_{Q^{2n-2}}^H)
   \cong\operatorname{IC}_{S_c}^H(-n).
\end{equation}
Furthermore, the Hodge rational homology level of $Z_{x_c}$ is
$\operatorname{HRH}(Z_{x_c})=n-c-2$.
\end{prop}

\begin{proof}
First suppose $d=2n-2$. For $0\leq i\leq d$, $i\ne n-1$,
write $x_i$ for the unique element of $\W^{\m}$ of length
$i$, with $x_0=e$, and
$x_{n-1}^{+},x_{n-1}^{-}$ for the two elements of length
$n-1$. The Bruhat order has the following Hasse diagram:
\begin{center}
\begin{tikzpicture}[x=1.45cm,y=.65cm,
  every node/.style={inner sep=2pt}]
\node (e) at (0,0) {$e$};
\node (xone) at (1,0) {$x_1$};
\node (ldots) at (2,0) {$\cdots$};
\node (left) at (3,0) {$x_{n-2}$};
\node (upper) at (4,.8) {$x_{n-1}^{+}$};
\node (lower) at (4,-.8) {$x_{n-1}^{-}$};
\node (right) at (5,0) {$x_n$};
\node (rdots) at (6,0) {$\cdots$};
\node (last) at (7,0) {$x_d$};
\draw[->] (e)--(xone);
\draw[->] (xone)--(ldots);
\draw[->] (ldots)--(left);
\draw[->] (left)--(upper);
\draw[->] (left)--(lower);
\draw[->] (upper)--(right);
\draw[->] (lower)--(right);
\draw[->] (right)--(rdots);
\draw[->] (rdots)--(last);
\end{tikzpicture}
\end{center}
Thus $GC_e^{\bullet}$ is
\[
0\longrightarrow N_e\longrightarrow\cdots
\longrightarrow N_{x_{n-2}}
\longrightarrow
N_{x_{n-1}^{+}}\oplus N_{x_{n-1}^{-}}
\longrightarrow N_{x_n}\longrightarrow\cdots
\longrightarrow N_{x_d}\longrightarrow0,
\]
with $N_x$ in degree $\ell(x)$.

By \cite{brenti}*{Theorems 4.2 and 4.3} and
Proposition~\ref{prop:am-weight-normalization}, for
$0\leq i\leq n-2$, the module $N_{x_i}$ has three Loewy
layers. Its socle is $L_{x_i}$, its head is $L_{x_{d-i}}$,
and its middle layer is
\[
\begin{cases}
L_{x_1},&i=0,\\
L_{x_{i+1}}\oplus L_{x_{d+1-i}},&1\leq i\leq n-3,\\
L_{x_{n-1}^{+}}\oplus L_{x_{n-1}^{-}}\oplus L_{x_{n+1}},
   &i=n-2.
\end{cases}
\]
The layers have weights $d+i$, $d+i+1$, and $d+i+2$,
respectively, from socle to head.

For $1\leq c\leq n-2$, the complex $GC_{x_c}^{\bullet}$
is obtained by deleting the terms of $GC_e^{\bullet}$
in degrees less than $c$. Since
$H^{\bullet}(GC_e^{\bullet})=L_e[0]$, we have
\[
H^q(GC_{x_c}^{\bullet})=0\quad(q\ne c),\qquad
H^c(GC_{x_c}^{\bullet})
=\operatorname{im}\!\left(
N_{x_{c-1}}\longrightarrow N_{x_c}
\right).
\]
The only common composition factors of $N_{x_{c-1}}$
and $N_{x_c}$ are $L_{x_c}$ and $L_{x_{d+1-c}}$,
in weights $d+c$ and $d+c+1$, respectively.
Theorem~\ref{thm:adjacentmaps} gives precisely
these two weight pieces of $H^c(GC_{x_c}^{\bullet})$.
Since $Z_{x_{d+1-c}}=S_c$, localization and
\eqref{eq:lc-tate-twist} give
\eqref{eq:quadric-weights}. In particular, none of the
singular Schubert varieties in type $\mathsf D$ is a
rational homology manifold.
The additional constituent $\operatorname{IC}_{S_c}^H(-n)$
starts in Hodge level
\[
\operatorname{codim}_{Q^d}(S_c)-n=n-c-1,
\]
so \eqref{eq:hrh-criterion} gives
$\operatorname{HRH}(Z_{x_c})=n-c-2$.

In type $\mathsf B$, write $x_i$ for the unique element
of length $i$, where $0\leq i\leq d=2n-1$.
The Bruhat order has Hasse diagram
\begin{center}
\begin{tikzpicture}[x=1.5cm,
  every node/.style={inner sep=2pt}]
\node (e) at (0,0) {$e=x_0$};
\node (xone) at (1,0) {$x_1$};
\node (dots) at (2,0) {$\cdots$};
\node (next) at (3,0) {$x_{d-1}$};
\node (last) at (4,0) {$x_d$};
\draw[->] (e)--(xone);
\draw[->] (xone)--(dots);
\draw[->] (dots)--(next);
\draw[->] (next)--(last);
\end{tikzpicture}
\end{center}
By the same formulas, $N_{x_i}$ has two Loewy layers
for $i<d$: socle $L_{x_i}$ in weight $d+i$ and head
$L_{x_{i+1}}$ in weight $d+i+1$; also $N_{x_d}=L_{x_d}$.
The proof is similar: truncation gives
$H^q(GC_{x_c}^{\bullet})=0$ for $q\ne c$ and
$H^c(GC_{x_c}^{\bullet})=L_{x_c}$ in weight $d+c$.
After localization, the latter is
$\operatorname{IC}_{Z_{x_c}}^H(-c)$, proving the
rational homology manifold assertion.
\end{proof}

\begin{remark}
For each singular $Z_{x_c}$, suitable coordinates on the
opposite big cell $U\cong\mathbb A^d$ of $Q^d$ give
\[
Z_{x_c}\cap U\cong
\mathbb A^{c-1}\times V(z_1^2+\cdots+z_{d-2c+2}^2).
\]
It is well-known that the hypersurface $V(z_1^2+\cdots+z_m^2)$ is a rational homology manifold if and only if $m$ is odd (equivalently, $d$ is odd; compare to Proposition \ref{prop:quadric-cohomology}). 
See \cite{MP1}*{Examples 20.4, 20.10} for discussion of the Hodge ideals of these hypersurfaces.
\end{remark}

\raggedbottom
\section{The exceptional pairs}\label{sec:exceptional-pairs}

In this section, we apply Theorem~\ref{thm:local-cohomology} to explicitly calculate the weight filtrations on local cohomology for the two exceptional
Hermitian pairs. As an application, in Section \ref{subsec:exceptional-affine}, we investigate homological invariants of orbit closures in the opposite big cell.

We enumerate the elements of $\W^{\m}$ following \cite[Tables 7.1, 7.2]{cisWeight}  (c.f. \cite{Collingwood1985}*{Tables 5--6 and Figures 2--3}). The tables in loc. cit. describe the weight filtrations on the modules $M_x$ (dually $N_x$). We display the Bruhat graphs in Figures 1 and 2 below, for convenience to the reader. In these graphs, an edge labeled $i$ indicates right multiplication
by $s_i$ from left to right, using Bourbaki's simple roots.

Table 1 below encodes the local cohomology modules corresponding to $(\mathsf{E}_6,\mathsf{D}_5)$, and Table 2 encodes those corresponding to $(\mathsf{E}_7,\mathsf{E}_6)$. Given $x\in \W^{\m}$, row $x$ of the table completely describes the weight filtration on $\mathscr H^{\bullet}_{Z_x}(\mathscr O_X^H)$. For every $x\in\W^{\m}$, the lowest weight piece is
\[
 \operatorname{Gr}^W_{d_X+\ell(x)}
 \mathscr H^{\ell(x)}_{Z_x}(\mathscr O_X^H)
 \cong \operatorname{IC}_{Z_x}^H(-\ell(x)).
\]
We omit this factor from the tables below.
An entry $(q,p;\{y_1,\cdots,y_m\})$ in row $x$ records the summands
\[
\operatorname{IC}_{Z_{y_1}}^H((d_X-\ell(y_1)-p)/2),\cdots,\operatorname{IC}_{Z_{y_m}}^H((d_X-\ell(y_m)-p)/2),
\]
of $\operatorname{Gr}^W_p\mathscr H^q_{Z_x}(\mathscr O_X^H)$. If the list $\{y_1,\cdots,y_m\}$ is a singleton, we omit the braces. 

We calculated Tables 1 and 2 using \cite[Tables 7.1, 7.2]{cisWeight}, Theorem \ref{thm:adjacentmaps}, and the Grothendieck--Cousin complex. The supplementary file \texttt{ExceptionalHermitian.m2} encodes
the aforementioned tables of \cite{cisWeight}, with the weight filtrations on $N_x$
normalized as in Proposition~\ref{prop:am-weight-normalization}.
The resulting local cohomology tables are calculated with and recorded in \texttt{grothendieckCousin.m2}. 

\subsection{The pair $(\mathsf E_6,\mathsf D_5)$}

Here $d_X=16$, and the identity has label $26$.
Every constituent in Table~\ref{tab:exceptional-e6}
has weight $q+17$. Thus
$\mathscr H^q_{Z_x}(\mathscr O_X^H)$ is pure of weight $q+17$
for $q>\ell(x)$, while the module in degree $\ell(x)$
has at most two nonzero weight pieces. We see that $\operatorname{HRH}(Z_x)\in\{0,1,2,+\infty\}$.

For example, $\ell(11)=9$, and
\[
 \operatorname{Gr}^W_p\mathscr H^9_{Z_{11}}(\mathscr O_X^H)
 \cong
 \begin{cases}
  \operatorname{IC}_{Z_{11}}^H(-9),&p=25,\\
  \operatorname{IC}_{Z_4}^H(-11),&p=26,\\
  0,&\text{otherwise}.
 \end{cases}
\]
The other nonzero local cohomology module is
\[
 \mathscr H^{10}_{Z_{11}}(\mathscr O_X^H)
 \cong \operatorname{IC}_{Z_1}^H(-13),
\]
which is pure of weight $27$.
All other local cohomology modules vanish.

\begin{table}[!htbp]
\begin{minipage}{\textwidth}
\centering
\caption{Additional factors for $(\mathsf E_6,\mathsf D_5)$.}
\label{tab:exceptional-e6}
\medskip
{\small
\setlength{\tabcolsep}{3pt}
\begin{minipage}[t]{0.49\textwidth}
\vspace{0pt}\centering
\begin{tabular}{@{}rrcl@{}}
\hline
$x$ & $\ell(x)$ & $\operatorname{HRH}(Z_x)$ & Additional $(q,p;y)$\\
\hline
0 & 16 & $+\infty$ & ---\\
1 & 15 & $+\infty$ & ---\\
2 & 14 & $+\infty$ & ---\\
3 & 13 & $+\infty$ & ---\\
4 & 12 & $+\infty$ & ---\\
5 & 12 & $+\infty$ & ---\\
6 & 11 & 0 & $(11,28;2)$\\
7 & 11 & $+\infty$ & ---\\
8 & 10 & 1 & $(10,27;1)$\\
9 & 10 & 0 & $(11,28;2)$\\
10 & 9 & 2 & $(9,26;0)$\\
11 & 9 & 0 & $(9,26;4),\quad (10,27;1)$\\
12 & 8 & $+\infty$ & ---\\
13 & 8 & 0 & $(9,26;\{0,4\})$\\
\hline
\end{tabular}
\end{minipage}\hfill
\begin{minipage}[t]{0.49\textwidth}
\vspace{0pt}\centering
\begin{tabular}{@{}rrcl@{}}
\hline
$x$ & $\ell(x)$ & $\operatorname{HRH}(Z_x)$ & Additional $(q,p;y)$\\
\hline
14 & 8 & 1 & $(10,27;1)$\\
15 & 7 & 0 & $(9,26;4)$\\
16 & 7 & 0 & $(7,24;9),\quad (9,26;0)$\\
17 & 6 & 0 & $(7,24;9)$\\
18 & 6 & 1 & $(6,23;7),\quad (9,26;0)$\\
19 & 5 & 0 & $(5,22;14),\quad (6,23;7)$\\
20 & 5 & 2 & $(9,26;0)$\\
21 & 4 & 1 & $(6,23;7)$\\
22 & 4 & 0 & $(5,22;14)$\\
23 & 3 & 0 & $(3,20;17)$\\
24 & 2 & 1 & $(2,19;15)$\\
25 & 1 & 2 & $(1,18;12)$\\
26 & 0 & $+\infty$ & ---\\
\hline
\end{tabular}
\end{minipage}
\par}
\end{minipage}
\end{table}

\begin{figure}[!htbp]
\centering
\begin{tikzpicture}[x=.0595\linewidth,y=.47cm,
  vertex/.style={circle,fill=white,inner sep=.5pt,minimum size=3.5mm,
    font=\fontsize{8}{9}\selectfont},
  rootlabel/.style={fill=white,inner sep=.3pt,
    font=\fontsize{6}{7}\selectfont}]
\node[vertex] (v26) at (0,0) {$26$};
\node[vertex] (v25) at (1,0) {$25$};
\node[vertex] (v24) at (2,0) {$24$};
\node[vertex] (v23) at (3,0) {$23$};
\node[vertex] (v22) at (4,1) {$22$};
\node[vertex] (v21) at (4,-1) {$21$};
\node[vertex] (v20) at (5,2) {$20$};
\node[vertex] (v19) at (5,0) {$19$};
\node[vertex] (v18) at (6,1) {$18$};
\node[vertex] (v17) at (6,-1) {$17$};
\node[vertex] (v16) at (7,0) {$16$};
\node[vertex] (v15) at (7,-2) {$15$};
\node[vertex] (v14) at (8,1) {$14$};
\node[vertex] (v13) at (8,-1) {$13$};
\node[vertex] (v12) at (8,-3) {$12$};
\node[vertex] (v11) at (9,0) {$11$};
\node[vertex] (v10) at (9,-2) {$10$};
\node[vertex] (v9) at (10,1) {$9$};
\node[vertex] (v8) at (10,-1) {$8$};
\node[vertex] (v7) at (11,2) {$7$};
\node[vertex] (v6) at (11,0) {$6$};
\node[vertex] (v5) at (12,1) {$5$};
\node[vertex] (v4) at (12,-1) {$4$};
\node[vertex] (v3) at (13,0) {$3$};
\node[vertex] (v2) at (14,0) {$2$};
\node[vertex] (v1) at (15,0) {$1$};
\node[vertex] (v0) at (16,0) {$0$};
\draw[line width=.3pt] (v1) -- node[rootlabel] {$6$} (v0);
\draw[line width=.3pt] (v2) -- node[rootlabel] {$5$} (v1);
\draw[line width=.3pt] (v3) -- node[rootlabel] {$4$} (v2);
\draw[line width=.3pt] (v4) -- node[rootlabel] {$2$} (v3);
\draw[line width=.3pt] (v5) -- node[rootlabel] {$3$} (v3);
\draw[line width=.3pt] (v6) -- node[rootlabel] {$3$} (v4);
\draw[line width=.3pt] (v6) -- node[rootlabel] {$2$} (v5);
\draw[line width=.3pt] (v7) -- node[rootlabel] {$1$} (v5);
\draw[line width=.3pt] (v8) -- node[rootlabel] {$4$} (v6);
\draw[line width=.3pt] (v9) -- node[rootlabel] {$1$} (v6);
\draw[line width=.3pt] (v9) -- node[rootlabel] {$2$} (v7);
\draw[line width=.3pt] (v10) -- node[rootlabel] {$5$} (v8);
\draw[line width=.3pt] (v11) -- node[rootlabel] {$1$} (v8);
\draw[line width=.3pt] (v11) -- node[rootlabel] {$4$} (v9);
\draw[line width=.3pt] (v12) -- node[rootlabel] {$6$} (v10);
\draw[line width=.3pt] (v13) -- node[rootlabel] {$1$} (v10);
\draw[line width=.3pt] (v13) -- node[rootlabel] {$5$} (v11);
\draw[line width=.3pt] (v14) -- node[rootlabel] {$3$} (v11);
\draw[line width=.3pt] (v15) -- node[rootlabel] {$1$} (v12);
\draw[line width=.3pt] (v15) -- node[rootlabel] {$6$} (v13);
\draw[line width=.3pt] (v16) -- node[rootlabel] {$3$} (v13);
\draw[line width=.3pt] (v16) -- node[rootlabel] {$5$} (v14);
\draw[line width=.3pt] (v17) -- node[rootlabel] {$3$} (v15);
\draw[line width=.3pt] (v17) -- node[rootlabel] {$6$} (v16);
\draw[line width=.3pt] (v18) -- node[rootlabel] {$4$} (v16);
\draw[line width=.3pt] (v19) -- node[rootlabel] {$4$} (v17);
\draw[line width=.3pt] (v19) -- node[rootlabel] {$6$} (v18);
\draw[line width=.3pt] (v20) -- node[rootlabel] {$2$} (v18);
\draw[line width=.3pt] (v21) -- node[rootlabel] {$5$} (v19);
\draw[line width=.3pt] (v22) -- node[rootlabel] {$2$} (v19);
\draw[line width=.3pt] (v22) -- node[rootlabel] {$6$} (v20);
\draw[line width=.3pt] (v23) -- node[rootlabel] {$2$} (v21);
\draw[line width=.3pt] (v23) -- node[rootlabel] {$5$} (v22);
\draw[line width=.3pt] (v24) -- node[rootlabel] {$4$} (v23);
\draw[line width=.3pt] (v25) -- node[rootlabel] {$3$} (v24);
\draw[line width=.3pt] (v26) -- node[rootlabel] {$1$} (v25);
\end{tikzpicture}
\caption{Bruhat order for $(\mathsf E_6,\mathsf D_5)$.}
\label{fig:exceptional-e6}
\end{figure}

\subsection{The pair $(\mathsf E_7,\mathsf E_6)$}

Here $d_X=27$, and the identity has label $55$.
The additional factors in Table~\ref{tab:exceptional-e7} below
have weights $q+28$ and $q+29$. For example, $\ell(49)=5$, and
\[
 \operatorname{Gr}^W_p\mathscr H^6_{Z_{49}}(\mathscr O_X^H)
 \cong
 \begin{cases}
  \operatorname{IC}_{Z_{40}}^H(-8),&p=34,\\
  \operatorname{IC}_{Z_5}^H(-15),&p=35,\\
  0,&\text{otherwise}.
 \end{cases}
\]
Each individual local cohomology module is multiplicity-free as a
$\mathscr D_X$-module in both exceptional pairs.
We see that $\operatorname{HRH}(Z_x)\in\{0,1,2,3,+\infty\}$.

\begin{table}[!htbp]
\begin{minipage}{\textwidth}
\centering
\caption{Additional factors for $(\mathsf E_7,\mathsf E_6)$.}
\label{tab:exceptional-e7}
\medskip
{\small
\setlength{\tabcolsep}{3pt}
\begin{minipage}[t]{0.49\textwidth}
\vspace{0pt}\centering
\begin{tabular}{@{}rrcl@{}}
\hline
$x$ & $\ell(x)$ & $\operatorname{HRH}(Z_x)$ & Additional $(q,p;y)$\\
\hline
0 & 27 & $+\infty$ & ---\\
1 & 26 & $+\infty$ & ---\\
2 & 25 & $+\infty$ & ---\\
3 & 24 & $+\infty$ & ---\\
4 & 23 & $+\infty$ & ---\\
5 & 22 & $+\infty$ & ---\\
6 & 22 & $+\infty$ & ---\\
7 & 21 & 0 & $(21,49;3)$\\
8 & 21 & $+\infty$ & ---\\
9 & 20 & 1 & $(20,48;2)$\\
10 & 20 & 0 & $(21,49;3)$\\
11 & 19 & 2 & $(19,47;1)$\\
12 & 19 & 0 & $(19,47;5),\quad (20,48;2)$\\
13 & 18 & 3 & $(18,46;0)$\\
14 & 18 & 0 & $(19,47;\{1,5\})$\\
15 & 18 & 1 & $(20,48;2)$\\
16 & 17 & $+\infty$ & ---\\
17 & 17 & 0 & $(18,46;0),\quad (19,47;5)$\\
18 & 17 & 0 & $(17,45;10),\quad (19,47;1)$\\
19 & 16 & 0 & $(19,47;5)$\\
20 & 16 & 0 & $(17,45;10),\quad (18,46;0)$\\
21 & 16 & 1 & $(16,44;8),\quad (19,47;1)$\\
22 & 15 & 0 & $(17,45;10)$\\
23 & 15 & 0 & $(15,43;15),\quad (16,44;8)$\\
 & &  & $(18,46;0)$\\
24 & 15 & 2 & $(19,47;1)$\\
25 & 14 & 0 & $(15,43;15),\quad (16,44;8)$\\
26 & 14 & 1 & $(16,44;8),\quad (18,46;0)$\\
27 & 14 & 0 & $(15,43;15),\quad (18,46;0)$\\
28 & 13 & 0 & $(13,41;21),\quad (16,44;8)$\\
29 & 13 & 0 & $(13,41;15),\quad (15,43;15)$\\
30 & 13 & 0 & $(13,41;20),\quad (18,46;0)$\\
\hline
\end{tabular}
\end{minipage}\hfill
\begin{minipage}[t]{0.49\textwidth}
\vspace{0pt}\centering
\begin{tabular}{@{}rrcl@{}}
\hline
$x$ & $\ell(x)$ & $\operatorname{HRH}(Z_x)$ & Additional $(q,p;y)$\\
\hline
31 & 12 & 1 & $(12,40;8),\quad (16,44;8)$\\
32 & 12 & 0 & $(12,40;\{18,22,24\})$\\
33 & 12 & 1 & $(12,40;17),\quad (18,46;0)$\\
34 & 11 & 0 & $(11,39;10),\quad (12,40;22)$\\
35 & 11 & 0 & $(11,39;\{14,19\}),\quad (12,40;24)$\\
36 & 11 & 2 & $(11,39;13),\quad (18,46;0)$\\
37 & 10 & 0 & $(10,38;\{12,29\}),\quad (10,39;15)$\\
 & &  & $(11,39;19)$\\
38 & 10 & 0 & $(10,38;\{11,16\}),\quad (12,40;24)$\\
39 & 10 & 3 & $(10,38;0),\quad (18,46;0)$\\
40 & 9 & 1 & $(9,37;5),\quad (11,39;19)$\\
41 & 9 & 0 & $(9,37;9),\quad (10,38;\{16,29\})$\\
 & &  & $(10,39;15)$\\
42 & 9 & 0 & $(9,37;1),\quad (12,40;24)$\\
43 & 8 & 0 & $(8,36;\{7,34\}),\quad (8,37;10)$\\
 & &  & $(10,38;16)$\\
44 & 8 & 0 & $(8,36;2),\quad (10,38;29)$\\
 & &  & $(10,39;15)$\\
45 & 7 & 1 & $(7,35;\{6,31\}),\quad (7,36;8)$\\
 & &  & $(10,38;16)$\\
46 & 7 & 0 & $(7,35;3),\quad (8,36;34)$\\
 & &  & $(8,37;10)$\\
47 & 6 & 2 & $(10,38;16)$\\
48 & 6 & 0 & $(6,34;\{4,40\}),\quad (6,35;5)$\\
 & &  & $(7,35;31),\quad (7,36;8)$\\
49 & 5 & 0 & $(6,34;40),\quad (6,35;5)$\\
50 & 5 & 1 & $(7,35;31),\quad (7,36;8)$\\
51 & 4 & 0 & $(4,32;46),\quad (4,33;3)$\\
52 & 3 & 1 & $(3,31;44),\quad (3,32;2)$\\
53 & 2 & 2 & $(2,30;42),\quad (2,31;1)$\\
54 & 1 & 3 & $(1,29;39),\quad (1,30;0)$\\
55 & 0 & $+\infty$ & ---\\
\hline
\end{tabular}
\end{minipage}
\par}
\end{minipage}
\end{table}

\begin{figure}[!htbp]
\centering
\begin{tikzpicture}[x=.0355\linewidth,y=.43cm,
  vertex/.style={circle,fill=white,inner sep=.5pt,minimum size=3.5mm,
    font=\fontsize{8}{9}\selectfont},
  rootlabel/.style={fill=white,inner sep=.3pt,
    font=\fontsize{6}{7}\selectfont}]
\node[vertex] (v55) at (0,0) {$55$};
\node[vertex] (v54) at (1,0) {$54$};
\node[vertex] (v53) at (2,0) {$53$};
\node[vertex] (v52) at (3,0) {$52$};
\node[vertex] (v51) at (4,0) {$51$};
\node[vertex] (v50) at (5,1) {$50$};
\node[vertex] (v49) at (5,-1) {$49$};
\node[vertex] (v48) at (6,0) {$48$};
\node[vertex] (v47) at (6,-2) {$47$};
\node[vertex] (v46) at (7,1) {$46$};
\node[vertex] (v45) at (7,-1) {$45$};
\node[vertex] (v44) at (8,2) {$44$};
\node[vertex] (v43) at (8,0) {$43$};
\node[vertex] (v42) at (9,3) {$42$};
\node[vertex] (v41) at (9,1) {$41$};
\node[vertex] (v40) at (9,-1) {$40$};
\node[vertex] (v39) at (10,4) {$39$};
\node[vertex] (v38) at (10,2) {$38$};
\node[vertex] (v37) at (10,0) {$37$};
\node[vertex] (v36) at (11,3) {$36$};
\node[vertex] (v35) at (11,1) {$35$};
\node[vertex] (v34) at (11,-1) {$34$};
\node[vertex] (v33) at (12,2) {$33$};
\node[vertex] (v32) at (12,0) {$32$};
\node[vertex] (v31) at (12,-2) {$31$};
\node[vertex] (v30) at (13,1.5) {$30$};
\node[vertex] (v29) at (13,0) {$29$};
\node[vertex] (v28) at (13,-1.5) {$28$};
\node[vertex] (v27) at (14,1.5) {$27$};
\node[vertex] (v26) at (14,0) {$26$};
\node[vertex] (v25) at (14,-1.5) {$25$};
\node[vertex] (v24) at (15,2) {$24$};
\node[vertex] (v23) at (15,0) {$23$};
\node[vertex] (v22) at (15,-2) {$22$};
\node[vertex] (v21) at (16,1) {$21$};
\node[vertex] (v20) at (16,-1) {$20$};
\node[vertex] (v19) at (16,-3) {$19$};
\node[vertex] (v18) at (17,0) {$18$};
\node[vertex] (v17) at (17,-2) {$17$};
\node[vertex] (v16) at (17,-4) {$16$};
\node[vertex] (v15) at (18,1) {$15$};
\node[vertex] (v14) at (18,-1) {$14$};
\node[vertex] (v13) at (18,-3) {$13$};
\node[vertex] (v12) at (19,0) {$12$};
\node[vertex] (v11) at (19,-2) {$11$};
\node[vertex] (v10) at (20,1) {$10$};
\node[vertex] (v9) at (20,-1) {$9$};
\node[vertex] (v8) at (21,2) {$8$};
\node[vertex] (v7) at (21,0) {$7$};
\node[vertex] (v6) at (22,1) {$6$};
\node[vertex] (v5) at (22,-1) {$5$};
\node[vertex] (v4) at (23,0) {$4$};
\node[vertex] (v3) at (24,0) {$3$};
\node[vertex] (v2) at (25,0) {$2$};
\node[vertex] (v1) at (26,0) {$1$};
\node[vertex] (v0) at (27,0) {$0$};
\draw[line width=.3pt] (v1) -- node[rootlabel] {$7$} (v0);
\draw[line width=.3pt] (v2) -- node[rootlabel] {$6$} (v1);
\draw[line width=.3pt] (v3) -- node[rootlabel] {$5$} (v2);
\draw[line width=.3pt] (v4) -- node[rootlabel] {$4$} (v3);
\draw[line width=.3pt] (v5) -- node[rootlabel] {$2$} (v4);
\draw[line width=.3pt] (v6) -- node[rootlabel] {$3$} (v4);
\draw[line width=.3pt] (v7) -- node[rootlabel] {$3$} (v5);
\draw[line width=.3pt] (v7) -- node[rootlabel] {$2$} (v6);
\draw[line width=.3pt] (v8) -- node[rootlabel] {$1$} (v6);
\draw[line width=.3pt] (v9) -- node[rootlabel] {$4$} (v7);
\draw[line width=.3pt] (v10) -- node[rootlabel] {$1$} (v7);
\draw[line width=.3pt] (v10) -- node[rootlabel] {$2$} (v8);
\draw[line width=.3pt] (v11) -- node[rootlabel] {$5$} (v9);
\draw[line width=.3pt] (v12) -- node[rootlabel] {$1$} (v9);
\draw[line width=.3pt] (v12) -- node[rootlabel] {$4$} (v10);
\draw[line width=.3pt] (v13) -- node[rootlabel] {$6$} (v11);
\draw[line width=.3pt] (v14) -- node[rootlabel] {$1$} (v11);
\draw[line width=.3pt] (v14) -- node[rootlabel] {$5$} (v12);
\draw[line width=.3pt] (v15) -- node[rootlabel] {$3$} (v12);
\draw[line width=.3pt] (v16) -- node[rootlabel] {$7$} (v13);
\draw[line width=.3pt] (v17) -- node[rootlabel] {$1$} (v13);
\draw[line width=.3pt] (v17) -- node[rootlabel] {$6$} (v14);
\draw[line width=.3pt] (v18) -- node[rootlabel] {$3$} (v14);
\draw[line width=.3pt] (v18) -- node[rootlabel] {$5$} (v15);
\draw[line width=.3pt] (v19) -- node[rootlabel] {$1$} (v16);
\draw[line width=.3pt] (v19) -- node[rootlabel] {$7$} (v17);
\draw[line width=.3pt] (v20) -- node[rootlabel] {$3$} (v17);
\draw[line width=.3pt] (v20) -- node[rootlabel] {$6$} (v18);
\draw[line width=.3pt] (v21) -- node[rootlabel] {$4$} (v18);
\draw[line width=.3pt] (v22) -- node[rootlabel] {$3$} (v19);
\draw[line width=.3pt] (v22) -- node[rootlabel] {$7$} (v20);
\draw[line width=.3pt] (v23) -- node[rootlabel] {$4$} (v20);
\draw[line width=.3pt] (v23) -- node[rootlabel] {$6$} (v21);
\draw[line width=.3pt] (v24) -- node[rootlabel] {$2$} (v21);
\draw[line width=.3pt] (v25) -- node[rootlabel] {$4$} (v22);
\draw[line width=.3pt] (v25) -- node[rootlabel] {$7$} (v23);
\draw[line width=.3pt] (v26) -- node[rootlabel] {$5$} (v23);
\draw[line width=.3pt] (v27) -- node[rootlabel] {$2$} (v23);
\draw[line width=.3pt] (v27) -- node[rootlabel] {$6$} (v24);
\draw[line width=.3pt] (v28) -- node[pos=.76,rootlabel] {$5$} (v25);
\draw[line width=.3pt] (v29) -- node[pos=.76,rootlabel] {$2$} (v25);
\draw[line width=.3pt] (v28) -- node[pos=.22,rootlabel] {$7$} (v26);
\draw[line width=.3pt] (v30) -- node[pos=.76,rootlabel] {$2$} (v26);
\draw[line width=.3pt] (v29) -- node[pos=.22,rootlabel] {$7$} (v27);
\draw[line width=.3pt] (v30) -- node[pos=.76,rootlabel] {$5$} (v27);
\draw[line width=.3pt] (v31) -- node[rootlabel] {$6$} (v28);
\draw[line width=.3pt] (v32) -- node[rootlabel] {$2$} (v28);
\draw[line width=.3pt] (v32) -- node[rootlabel] {$5$} (v29);
\draw[line width=.3pt] (v32) -- node[rootlabel] {$7$} (v30);
\draw[line width=.3pt] (v33) -- node[rootlabel] {$4$} (v30);
\draw[line width=.3pt] (v34) -- node[rootlabel] {$2$} (v31);
\draw[line width=.3pt] (v34) -- node[rootlabel] {$6$} (v32);
\draw[line width=.3pt] (v35) -- node[rootlabel] {$4$} (v32);
\draw[line width=.3pt] (v35) -- node[rootlabel] {$7$} (v33);
\draw[line width=.3pt] (v36) -- node[rootlabel] {$3$} (v33);
\draw[line width=.3pt] (v37) -- node[rootlabel] {$4$} (v34);
\draw[line width=.3pt] (v37) -- node[rootlabel] {$6$} (v35);
\draw[line width=.3pt] (v38) -- node[rootlabel] {$3$} (v35);
\draw[line width=.3pt] (v38) -- node[rootlabel] {$7$} (v36);
\draw[line width=.3pt] (v39) -- node[rootlabel] {$1$} (v36);
\draw[line width=.3pt] (v40) -- node[rootlabel] {$5$} (v37);
\draw[line width=.3pt] (v41) -- node[rootlabel] {$3$} (v37);
\draw[line width=.3pt] (v41) -- node[rootlabel] {$6$} (v38);
\draw[line width=.3pt] (v42) -- node[rootlabel] {$1$} (v38);
\draw[line width=.3pt] (v42) -- node[rootlabel] {$7$} (v39);
\draw[line width=.3pt] (v43) -- node[rootlabel] {$3$} (v40);
\draw[line width=.3pt] (v43) -- node[rootlabel] {$5$} (v41);
\draw[line width=.3pt] (v44) -- node[rootlabel] {$1$} (v41);
\draw[line width=.3pt] (v44) -- node[rootlabel] {$6$} (v42);
\draw[line width=.3pt] (v45) -- node[rootlabel] {$4$} (v43);
\draw[line width=.3pt] (v46) -- node[rootlabel] {$1$} (v43);
\draw[line width=.3pt] (v46) -- node[rootlabel] {$5$} (v44);
\draw[line width=.3pt] (v47) -- node[rootlabel] {$2$} (v45);
\draw[line width=.3pt] (v48) -- node[rootlabel] {$1$} (v45);
\draw[line width=.3pt] (v48) -- node[rootlabel] {$4$} (v46);
\draw[line width=.3pt] (v49) -- node[rootlabel] {$1$} (v47);
\draw[line width=.3pt] (v49) -- node[rootlabel] {$2$} (v48);
\draw[line width=.3pt] (v50) -- node[rootlabel] {$3$} (v48);
\draw[line width=.3pt] (v51) -- node[rootlabel] {$3$} (v49);
\draw[line width=.3pt] (v51) -- node[rootlabel] {$2$} (v50);
\draw[line width=.3pt] (v52) -- node[rootlabel] {$4$} (v51);
\draw[line width=.3pt] (v53) -- node[rootlabel] {$5$} (v52);
\draw[line width=.3pt] (v54) -- node[rootlabel] {$6$} (v53);
\draw[line width=.3pt] (v55) -- node[rootlabel] {$7$} (v54);
\end{tikzpicture}
\caption{Bruhat order for $(\mathsf E_7,\mathsf E_6)$.}
\label{fig:exceptional-e7}
\end{figure}

\subsection{Restricting to the opposite big cell}
\label{subsec:exceptional-affine}

We restrict the calculations above to the opposite big cell and compare to some previously studied ideals. Write $U_6$ and $U_7$ for the opposite big cells corresponding
to $(\mathsf E_6,\mathsf D_5)$ and $(\mathsf E_7,\mathsf E_6)$,
respectively. We endow $U_i$ with the coordinates of Filippini--Torres--Weyman \cite{FTW}, though our indexing of $\W^{\m}$ is different from loc. cit. For the translation, compare Figure 1 to \cite[Figure 1]{FTW} and Figure 2 to \cite[Figure 2]{FTW}.

\smallskip
\noindent\textit{The Freudenthal cubic.}
Let $f$ be the Freudenthal cubic on $U_7\cong \mathbb A^{27}$, which satisfies $V(f)=Z_{54}\cap U_7$. The ideal generated by $f$ is denoted by $J_{55}$ in \cite{FTW}. Row $54$ of Table~\ref{tab:exceptional-e7} recovers
the weight filtration on
$\mathcal O_{\mathbb A^{27}}[f^{-1}]$
computed by L\H{o}rincz--Yang
\cite{LYsemi}*{Section~5.4}, who also determine its Hodge
filtration. Row $39$ of Table 2 gives the weight filtration on local cohomology supported in the singular locus.

\smallskip
\noindent\textit{A variety of complexes.}
Let $A$ and $B$ be generic matrices of sizes $2\times3$ and
$3\times2$, respectively, and put
\[
 N=\{(A,B)\mid AB=0,\ \operatorname{rank}B\leq1\},\qquad
 N'=\{(A,B)\mid B=0,\ \operatorname{rank}A\leq1\}
 \subseteq\mathbb A^{12}.
\]
The variety $N$ is an irreducible component of the nullcone
$\{(A,B)\mid AB=0\}$ for the action
$g\cdot(A,B)=(Ag^{-1},gB)$ of $\mathrm{GL}_3(\mathbb{C})$. The coordinates in
\cite{FTW}*{Sections~3.10 and~3.20} identify
\[
 Z_{17}\cap U_6=N\times\mathbb A^2\times\{0\},\qquad
 Z_9\cap U_6=N'\times\mathbb A^2\times\{0\}.
\]
It follows from Table~\ref{tab:exceptional-e6} that $\mathscr H_N^q(\mathcal O_{\mathbb A^{12}}^H)$ is described as
\begin{equation}
 \mathscr H_N^q(\mathcal O_{\mathbb A^{12}}^H)
 \cong
 \begin{cases}
  \operatorname{IC}_N^H(-4),&q=4,\\
  \operatorname{IC}_{N'}^H(-7),&q=5,\\
  0,&\text{otherwise}.
 \end{cases}
\end{equation} 
For arithmetic rank and local cohomological dimension of
determinantal nullcones, see Jeffries--Pandey--Singh--Walther
\cite{JPSW}.

\smallskip
\noindent\textit{A Huneke--Ulrich ideal.}
Let $M$ be a generic $6\times6$ skew-symmetric matrix and $v$
a generic row vector of length six. Set
\[
 Y=\{(M,v)\mid vM=0,\ \operatorname{Pf}(M)=0\},\qquad
 S=\{(M,v)\mid v=0,\ \operatorname{rank}M\leq2\}
 \subseteq\mathbb A^{21}.
\]
The defining ideal $I$ of $Y$ is a seven-generated
Huneke--Ulrich ideal of codimension five
\cite{FTW}*{Section~4.6}.
The patches $Z_{50}\cap U_7$, $Z_{31}\cap U_7$, and $Z_8\cap U_7$
identify with $Y\times\mathbb A^6$, $S\times\mathbb A^6$,
and $\{0\}\times\mathbb A^6$, respectively.
Removing this common affine factor lowers the weights by six,
so Table~\ref{tab:exceptional-e7} gives
\[
 \mathscr H_Y^5(\mathcal O_{\mathbb A^{21}}^H)
 \cong\operatorname{IC}_Y^H(-5),
 \qquad
 \operatorname{Gr}_p^W
 \mathscr H_Y^7(\mathcal O_{\mathbb A^{21}}^H)
 \cong
 \begin{cases}
  \operatorname{IC}_S^H(-10),&p=29,\\
  \operatorname{IC}_{\{0\}}^H(-15),&p=30,\\
  0,&\text{otherwise},
 \end{cases}
\]
and $\mathscr H_Y^q(\mathcal O_{\mathbb A^{21}}^H)=0$
for $q\notin\{5,7\}$.
In particular, the Huneke--Ulrich ideal $I$ has arithmetic
rank seven.

\section*{Acknowledgments}
Experiments with Macaulay2 \cite{M2} and ChatGPT \cite{chat} have provided valuable insights.
 We are grateful to Sam Evens, Andr\'{a}s L\H{o}rincz, Claudiu Raicu, and Vic Reiner for advice and useful conversations regarding this project. We thank Brian Boe for inspirational correspondence. The author acknowledges the support of the National Science Foundation under Grant No. DMS-2601624.

\end{document}